\documentclass[11pt]{article}
\usepackage{a4wide}

\usepackage{graphicx}
\usepackage{amsmath}
\usepackage{amsthm}
\usepackage{amssymb}
\usepackage{hyperref}

\newtheorem{theorem}{Theorem}
\newtheorem*{theorem*}{Theorem}
\newtheorem{lemma}{Lemma}
\newtheorem*{lemma*}{Lemma}

\newtheorem{example}{Example}
\newtheorem*{example*}{Example}
\newtheorem{corollary}{Corollary}
\newtheorem*{corollary*}{Corollary}
\newtheorem{conjecture}{Conjecture}
\theoremstyle{remark}

\newtheorem*{remark*}{Remarks} %%Fix this if I use it as a preamble for other documents

\newcommand{\brac}[1]{\left( #1 \right)}
\newcommand{\prin}[1]{\left< #1 \right>}
\newcommand{\squar}[1]{\left[ #1 \right]}
\newcommand{\abs}[1]{\left| #1 \right|}
\newcommand{\curly}[1]{\left\{ #1 \right\}}
\newcommand{\comma}{, \, }

\newcommand{\NN}{\mathbb{N}}
\newcommand{\ZZ}{\mathbb{Z}}
\newcommand{\QQ}{\mathbb{Q}}
\newcommand{\R}{\mathbb{R}}
\newcommand{\RR}{\mathbb{R}}
\newcommand{\CC}{\mathbb{C}}
\newcommand{\Of}[1]{\mathcal{O}_{#1}}
\newcommand{\id}[1]{\mathfrak{#1}}
\newcommand{\house}[1]{{%
    \setbox0=\hbox{$#1$}
    \vrule height \dimexpr\ht0+1.4pt width .4pt depth \dp0\relax
\vrule height \dimexpr\ht0+1.4pt width \dimexpr\wd0+2pt depth \dimexpr-\ht0-1pt\relax
    \llap{$#1$\kern1pt}
    \vrule height \dimexpr\ht0+1.4pt width .4pt depth \dp0\relax}}
\newcommand{\Normf}[2]{{\text{N}_{ #1 / \QQ}}\brac{#2}}

\newcommand{\ordid}[1]{\text{ord}_{\mathfrak{p}}\brac{#1}}
\newcommand{\h}[1]{h\brac{#1}}
\newcommand{\ordp}[1]{\text{ord}_{p} \brac{#1}}

\title{On the $abc$ Conjecture in Algebraic Number Fields}
\author{Andrew Scoones}
\date{}

\begin{document}

\maketitle
\bigskip

\begin{abstract}
    In this paper we prove a weak form of the $abc$ Conjecture generalised to algebraic number fields. Given algebraic integers $a,\,b,\,c$ in a number field $K$ satisfying $a+b=c$, we give an upper bound for the logarithm of the projective height $H_{L}(a,\,b,\,c)$ in terms of norms of prime ideals dividing $abc \Of{L}$, where $L$ is the Hilbert Class Field of $K$. In many cases this allows us to give a bound in terms of the modified radical $G:=G(a,\,b,\,c)$ as given by Masser in \cite{masser2002abc}. Furthermore, by employing a recent result by Le Fourn, our estimates imply the upper bound
    \[\log H_{L}\brac{a,\,b,\,c}< G^{\frac{1}{3}+\mathcal{C}\frac{\log \log \log G}{\log \log G}},\]
    where $\mathcal{C}$ is an effectively computable constant. Further, given conditions on the largest prime ideal dividing $abc \Of{L}$, we obtain a sub-exponential bound for $H_{L}(a,\,b,\,c)$ in terms of the radical. As a consequence of our results, we will give an application to the effective Skolem-Mahler-Lech problem and give an improvement to a result by Lagarias and Soundararajan on the XYZ Conjecture given in \cite{lagarias2011smooth}. 
\end{abstract}

\section{Introduction}

Let $a,\,b,\, c:=a+b$ be positive, pairwise coprime integers and define the radical 
\begin{equation*}
G(a,\,b,\,c)=G = \prod_{\substack{p \mid abc \\ p \text{ a prime}}} p.
\end{equation*}
 In 1988, based on a conjecture of Szpiro about elliptic curves and on the Mason-Stothers Theorem in function fields \cite{mason1984diophantine}, Osterl\'e conjectured  that for all positive integers as above, there exists a positive constant $\mathcal{C}_{1}$ such that $c<G^{\mathcal{C}_{1}}$ \cite{oesterle1988nouvelles}. Further, Masser conjectured a stronger statement, that for all positive $\epsilon$ there exists a constant $\mathcal{C}_{2}\brac{\epsilon}$ such that $c < \mathcal{C}_{2}\brac{\epsilon}G^{1+\epsilon}$ \cite{masser1985open}. While both of these conjectures are referred to as the $abc$ conjecture, generally the second form by Masser is focused on in the literature. These conjectures have far-reaching implications across a range of topics; see \cite{browkin2000abc} and the references within. We also refer the reader to Chapter 14 of \cite{bombieri2007heights} and Chapter 5 of \cite{vojta2006diophantine} for a discussion of Vojta's conjectures, a generalisation of the second formulation of the $abc$ conjecture given above. 

In \cite{StewartYu}, Stewart and Yu prove that there exists an effectively computable positive constant $\mathcal{C}_{3}$ such that for all positive integers $a,\,b,\, c=a+b$ with $(a,\,b,\,c)=1$ and $c>2$, 
\[\log c < G^{\frac{2}{3}+\frac{\mathcal{C}_{3}}{\log \log G}}.\]
In \cite{stewart2001abc}, they were able to improve this result to \[\log{c}<\mathcal{C}_{4}G^{\frac{1}{3}}\brac{\log G}^{3}.\] To do this required the use of Yu's work  extending lower bounds for linear forms in logarithms to the $p$-adic setting \cite{yu1989linear}\cite{yu1990linear}. We note Yu further improved these bounds in a series of papers \cite{yu1994linear} \cite{yu1998p}\cite{yu1999p}\cite{yu2007p}; indeed, we will use results from \cite{yu2007p}, which make use of group varieties to strengthen the relevant bounds. 

Much work has also been done generalising the $abc$ conjecture to algebraic number fields. Browkin discusses this direction of research in \cite{browkin2006abc}, while in \cite{masser2002abc} Masser discusses some issues regarding adapting the radical $G$ to the case of number fields. In \cite{gyHory2008abc}, Gy\"{o}ry shows that given a number field $K$ and $a,\,b,\, c \in K^{*}$ with $a+b+c=0$ and any $\epsilon > 0$, there is an effectively computable $\mathcal{C}_{5}\brac{\epsilon}$ such that
\[\log\brac{H_{K}\brac{a,\,b,\,c}} < \mathcal{C}_{5}N_{K}\brac{a,\,b,\,c}^{1+\epsilon},\]
with \[H_{K}\brac{a,\,b,\,c} = \prod_{\upsilon \in M_{K}} \max \brac{\abs{a}_{\upsilon},\, \abs{b}_{\upsilon},\, \abs{c}_{\upsilon}},\] where $M_{K}$ is the set of normalised places of $K$ and \[N_{K}\brac{a,\,b,\,c} = \prod_{\upsilon} \text{Nm}_{\QQ}^{K}\brac{\id{p}}^{\ordid{p}},\] where $p$ is a rational prime such that $\id{p}$ lies over $p$, $\id{p}$ is the prime ideal of $\Of{K}$ corresponding to $\upsilon \in M_{K}$ and $\upsilon$ is taken over all finite places such that $\abs{a}_{\upsilon},\, \abs{b}_{\upsilon},\, \abs{c}_{\upsilon}$ are not all equal. We note that this is the same as the modified support (1.11) of \cite{masser2002abc}. Recall that the norm of a prime ideal $\id{p}$ of the ring of integers $\Of{K}$ of a number field $K$ is defined to be
\[\textrm{Nm}_{QQ}^{K}\brac{\id{p}}=p^{f_{\id{p}}},\]
where $f_{\id{p}}$ is the inertia degree of $\id{p}$ over $p$ \cite{neukirch2013algebraic}; that is
\[f_{\id{p}}:=\squar{\Of{K}/\id{p}:\ZZ/p}.\]

Initially we introduce some notation we will use throughout this article. Let $K$ be a number field of degree $d$ and let $a,\,b,\, c \in \Of{K} \backslash \curly{0}$ be such that $a+b+c=0$. Further, assume that $a\Of{K},\, b\Of{K}$ and $c\Of{K}$ are pairwise coprime; that is $a\Of{K} + b\Of{K} = \Of{K}$, and similarly for all other pairs.  Let $L=HCF(K)$ be the Hilbert Class Field of $K$ (that is, the maximal abelian unramified extension of $K$ \cite{childress2008class}) and let 
\[G=\prod_{\substack{\id{P} \text{ prime ideal} \\ \id{P} \subset \Of{L} \\ \id{P} \mid \brac{abc}\Of{L} }} \text{Nm}_{\QQ}^{L} \brac{\id{P}}.\] Let $\id{p}_{a}$ be the prime ideal of $\Of{L}$ of greatest norm dividing $a\Of{L}$, and similarly for $\id{p}_{b}$ and $\id{p}_{c}$. If $a$ is a unit, then we write that $\id{p}_{a}=1$ with norm 1, and similarly for $b$ and $c$. Write $\id{p}_{\max}$ for the prime ideal of $\Of{L}$ of greatest norm dividing $G$. A priori, this is equal to one of $\id{p}_{a},\, \id{p}_{b},\, \id{p}_{c}$.

Let 
\[h(x)=\sum_{\upsilon \in M_{F}} \log^{+} \abs{x}_{\upsilon}\]
where $M_{F}$ is the set of places of the number field $F$ normalised so they satisfy the product formula \cite{bombieri2007heights}, and $\log^{+}(\alpha)=\max (\log \alpha,\, 0)$. Also, let 
\[H_{F}\brac{x_{1},\,\dots,\,x_{n}} = \prod_{\upsilon \in M_{F}} \max \curly{\abs{x_{1}}_{\upsilon},\, \dots,\, \abs{x_{n}}_{\upsilon}}.\] 

It is worth pointing out that $H_{F}\brac{x_{1},\dots,\, x_{n}}$ is the projective height, so it gives the same value for any representative of $(x_{1},\dots,\, x_{n}) \in \mathbb{P}^{n-1}\brac{F}$. Explicitly, this means that for any $\brac{a,\,b,\,c} \in \mathbb{P}^{2}\brac{F}$ and any $k \in F^{\times}$ we have that
\begin{equation}\label{property of H}
H_{F}\brac{a,\,b,\,c} = H_{F}\brac{ka,\,kb,\,kc}.
\end{equation}
In particular, since in the set up of this article $c \neq 0$, we have that \[H_{L}\brac{a,\,b,\,c}=H_{L}\brac{\frac{a}{c},\, \frac{b}{c},\,1}.\] 

We will generally be considering the height over the Hilbert Class Field $L$. In this case, as $\squar{K:\QQ}=d$, \[h(x)= d h_{K}\log H_{L}(1,\,x),\]  where $h_{K}$ is the class number of $K$. This follows as $\squar{L:K}=h_{K}$\cite{bombieri2007heights} \cite{Waldschmidt}

We note that for any $x,\, y,\, z \in F$ where $F$ is an algebraic number field of degree $d$, 
\begin{align}\label{relation between H and h}
\log H_{F}(x,\,y,\,z) &= \log H_{F}\brac{\frac{x}{z},\,\frac{y}{z},\,1} \nonumber \\
&\leq 2d \max \brac{\h{\frac{x}{z}},\, \h{\frac{y}{z}}}.
\end{align}
This follows directly from (4.3) of \cite{gyHory2008abc}. Furthermore, we will show in Section 3 that we can write $a=u_{a}a'$ where $u_{a}$ is a unit such that
\[\mathcal{C}_{6}\log \abs{\Normf{L}{a'}}\leq \h{a'} \leq \mathcal{C}_{7}\log \abs{\Normf{L}{a'}},\]
where $\mathcal{C}_{6},\, \mathcal{C}_{7}$ are computable constants, and similarly for $b$ and $c$. We assume without loss of generality that 
\begin{equation} \label{double inequality}
    h(a')\leq h(b')\leq h(c').
\end{equation}

We initially prove the following main theorem.

\begin{theorem}\label{main theorem}
Given the set up above, relabeling $a,\,b$ and $c$ if necessary to satisfy \eqref{double inequality}, there exists an effectively computable constant $\mathcal{C}_{8}$ depending only on the field $K$ such that
\begin{align}
\log H_{L}(a,\,b,\,c) &<  \brac{\text{Nm}_{\QQ}^{L}\brac{\id{p}_{a }}\text{Nm}_{\QQ}^{L}\brac{\id{p}_{b}} \text{Nm}_{\QQ}^{L}\brac{\id{p}_{c}}^{2}\max \curly{\text{Nm}_{\QQ}^{L}\brac{\id{p}_{b}},\,\text{Nm}_{\QQ}^{L}\brac{\id{p}_{c}} }}^{\frac{1}{3}} G^{\mathcal{C}_{8}\frac{\log \log \log G}{\log \log G}} .
\end{align} 
\end{theorem}

We will then give various corollaries to put the product of norms of prime ideals in terms of the radical $G$, namely corollaries 1-7. Importantly, in Corollary 7 we will give conditions that allow us to attain a sub-exponential bound.

In later parts we will give related results that can be easier to manipulate due to fewer prime ideals on the right hand side of the inequality, attaining the following theorem.

\begin{theorem}\label{thm 1}
Given the set up above, there exists an effectively computable number $\mathcal{C}_{9}$ depending only on the field $K$ such that
\begin{align}
\log H_{L}(a,\,b,\,c) &<  \brac{\text{Nm}_{\QQ}^{L}\brac{\id{p}_{b}} \text{Nm}_{\QQ}^{L}\brac{\id{p}_{c}}^{2}}^{\frac{1}{2}} G^{\mathcal{C}_{9}\frac{\log \log \log G}{\log \log G}} \nonumber \\
 &= \text{Nm}_{\QQ}^{L}\brac{\id{p}_{b}}^{\frac{1}{2}} \text{Nm}_{\QQ}^{L}\brac{\id{p}_{c}} G^{\mathcal{C}_{9}\frac{\log \log \log G}{\log \log G}}.
\end{align} 
\end{theorem}

We will then deduce corollaries 8-13, again giving conditions in Corollary 10 that give a sub-exponential bound in terms of the radical $G$.

We will then discuss how exploiting a method of Le Fourn \cite{le2020tubular} enables us to reduce the dependency on prime ideals to give the following result with no further conditions:

\begin{theorem}
Given the set up above, there exists an effectively computable constant $\mathcal{C}_{10}$ depending on $K$ such that
\begin{align}
\log H_{L}(a,\,b,\,c) &<  \brac{\text{Nm}_{\QQ}^{L}\brac{\id{p}_{a }}\text{Nm}_{\QQ}^{L}\brac{\id{p}_{b}} \text{Nm}_{\QQ}^{L}\brac{\id{p}_{c}} \text{Nm}_{\QQ}^{L}\brac{\id{p}_{c}'} \text{Nm}_{\QQ}^{L}\brac{\id{q}}}^{\frac{1}{3}} G^{\mathcal{C}_{10}\frac{\log \log \log G}{\log \log G}},
\end{align} 
where $\id{p}_{c}'$ is the prime ideal of third largest norm dividing $c\Of{L}$ and $\id{q}$ is the prime ideal of $\Of{L}$ of third largest norm dividing $bc\Of{L}$.
\end{theorem}
 
 From here we will deduce that
 \begin{align}
\log H_{L}(a,\,b,\,c) &<  G^{\frac{1}{3}+ \mathcal{C}_{11}\frac{\log \log \log G}{\log \log G}}.
\end{align} 

 The results given in this paper, in particular Theorem 3, allow us to give a new method of solving the effective Skolem-Mahler-Lech problem \cite{ouaknine2012decision} of order 3. Additionally, we use Corollary 10 to expand on results by Lagarias and Soundararajan regarding smooth solutions to the $abc$ equation \cite{lagarias2011smooth}. We briefly discuss both these problems here.

First we discuss the effective Skolem-Mahler-Lech problem. The problem is, given a linear recurrence sequence, to decide whether said sequence contains zeroes.

We recall that a linear recurrence sequence is a sequence $\brac{a_{x}}$ of elements of a commutative ring with 1, $R$, satisfying a homogeneous linear recurrence relation 
\[a_{x+n}=c_{1}a_{x+n-1}+\cdots + c_{n} a_{x},\]
where $c_{1},\dots,\, c_{n} \in R$ \cite{everest2003recurrence}. We note, we will take $R$ to be an algebraic number field.

The Skolem-Mahler-Lech Theorem is as follows.

\begin{theorem}[Skolem-Mahler-Lech]
If a sequence of numbers satisfies a linear recurrence relation over a field of characteristic zero, the the zeroes of this sequence can be decomposed into the union of a finite set, and finitely many arithmetic progressions.
\end{theorem}

More comments on this theorem and a proof for sequences defined over the rationals is given in \cite{everest2003recurrence}, along with further references.

\begin{remark*}
We note that there exists an algorithm to tell us if there are infinitely many zeroes, and if so to find the decomposition of these zeros into periodic sets guaranteed to exist by the Skolem–Mahler–Lech Theorem \cite{berstel1976deux}. The effective Skolem-Mahler-Lech problem then is to find an algorithm to determine whether there are any non-periodic zeroes in a given linear recurrence sequence, importantly in the case where these are the only zeroes \cite{ouaknine2012decision}. This would allow us to effectively answer whether a given linear recurrence relation contains any zeroes.
\end{remark*}

Recall, given a linear recurrence relation \[a_{x+n}=c_{1}a_{x+n-1}+\cdots + c_{n} a_{x},\]
and initial terms $a_{1},\dots,\, a_{n} \in R$, we can find a formula for the $m$'th term. Given the recurrence relation, we find the characteristic polynomial
\[f(X)=X^{n}-c_{1}X^{n-1}-\cdots -c_{n-1}X-c_{n},\]
with roots $r_{1},\dots,\, r_{l}$ with multiplicities $m_{1},\dots,\, m_{l}$ respectively. The $x$'th term of the sequence then is given by
\[a_{x}=g_{1}(x)r_{1}^{x}+\cdots +g_{l}(x)r_{l}^{x},\]
where $g_{i}(x)$ are polynomials with $\deg \brac{g_{i}}\leq m_{i}-1$ which depend on the initial values $a_{1},\, \dots,\, a_{n}$.

In section 8 we will show how Theorem 3 gives a new method of determining if a linear recurrence sequence of order 3 contains zeroes. For further references of preexisting methods and results on this problem and related problems we refer the reader to \cite{halava2005skolem} \cite{sha2019effective} \cite{ostafe2020skolem} and the references contained within them.

We now discuss the smooth $abc$ Conjecture, also referred to as the $xyz$ conjecture given by Lagarias and Soundararajan in \cite{lagarias2011smooth}. Given a triple $a,\,b,\,c:=a+b \in \NN$, define the smoothness of the triple
\[S(a,\,b,\,c):=\max \curly{p\,:\, p\mid abc}.\]
In \cite{lagarias2011smooth}, Lagarias and Soundararajan give the following conjecture, which they refer to as the $xyz$ conjecture.

\begin{conjecture}[$xyz$ conjecture]
There exists a positive constant $\kappa$ such that the following hold.

a) For each $\epsilon>0$ there are only finitely many integer solutions $\brac{X,\,Y,\,Z}$ to the equation $X+Y=Z$ with
$\brac{X,\,Y,\,Z}=1$ and
\[S\brac{X,\,Y,\,Z}<\brac{\log H\brac{X,\,Y,\,Z}}^{\kappa-\epsilon}.\]

b) For each $\epsilon>0$ there are infinitely many integer solutions $\brac{X,\,Y,\,Z}$ to the equation $X+Y=Z$ with $\brac{X,\,Y,\,Z}=1$ and
\[S\brac{X,\,Y,\,Z}<\brac{\log H\brac{X,\,Y,\,Z}}^{\kappa+\epsilon}.\]
\end{conjecture}
When a triple $\brac{X,\,Y,\,Z}$ satisfies $X+Y=Z$ and $\brac{X,\,Y,\,Z}=1$, we will call the triple a primitive solution. 
Lagarias and Soundararajan go on to conjecture that $\kappa = \frac{3}{2}$. They note however that to prove the above conjecture, one need only prove that there exists a $\kappa_{0}>0$ satisfying part a) and a $\kappa_{1}< \infty$ satisfying part b). As a) and b) are independent, monotonicity would then imply the existence of a unique constant $\kappa$.

Lagarias and Soundararajan prove part b) assuming the Generalised Riemann Hypothesis, and show that the $abc$ conjecture implies part a). Further, in Corollary 1 of \cite{harper2016minor}, Harper is able to show unconditionally that the $xyz$-smoothness exponent $\kappa$ is finite, and showed that part b) of the conjuecture holds.

Unconditionally, Lagarias and Soundararajan are able to give the following result.
\begin{theorem*}[Theorem 2.2 of \cite{lagarias2011smooth}]
For each $\epsilon > 0$ there are only finitely many solutions to $X+Y=Z$ satisfying $\brac{X,\,Y,\,Z}=1$ and 
\[S\brac{X,\,Y,\,Z}\leq \brac{3-\epsilon}\log \log H\brac{X,\,Y,\,Z}.\]
\end{theorem*}
The proof of this, and the improvement we will give below depend heavily on Northcott's Theorem. We recall that Northcott's Theorem tells us that for any given number field $K$ and $B\in \RR,\, B>0$, the set
\[\curly{\brac{X,\,Y,\,Z} \in K^{3} \,:\, H\brac{X,\,Y,\,Z}<B}\]
is finite \cite{bombieri2007heights}.

Using results from this paper, we will improve this bound with the following theorem. 

\begin{theorem}\label{XYZ improvement}
Let $\phi:\RR\rightarrow\RR$ be a function such that $\phi\brac{x}< \log \log x$ with 
\[\lim_{x \rightarrow +\infty}\phi\brac{x}= +\infty.\]
Then there are finitely many solutions to $X+Y=Z$ satisfying $\brac{X,\,Y,\,Z}=1$ and
\begin{equation}\label{S assumption}
S\brac{X,\,Y,\,Z}\leq \log \log H\brac{X,\,Y,\,Z}\frac{\log \log \log H\brac{X,\,Y,\,Z}}{\log \log \log \log H\brac{X,\,Y,\,Z} \phi\brac{\log \log H\brac{X,\,Y,\,Z}}}.    
\end{equation}

\end{theorem}

We note that this result implies that there are only finitely many primitive integer triples $\brac{X,\,Y,\,Z}$ satisfying $X+Y=Z$ with 
\[S\brac{X,\,Y,\,Z} < c \log \log H\brac{X,\,Y,\,Z}\]
for any constant $c \in \RR,\, c>0$. This is because for any such $c$, there is a value $H$ such that for any $H\brac{X,\,Y,\,Z}>H,
$ \[\frac{\log \log \log H\brac{X,\,Y,\,Z}}{\log \log \log \log H\brac{X,\,Y,\,Z} \phi\brac{\log \log H\brac{X,\,Y,\,Z}}} > c,\]
and by Northcott's Theorem, there are only finitely many triples $\brac{X,\,Y,\,Z}$ satisfying $H\brac{X,\,Y,\,Z}<H$.
The statement above then follows from Theorem~\ref{XYZ improvement}, and along with the result given in Theorem~\ref{XYZ improvement} is also an improvement on the $3-\epsilon$ in Theorem 2.2 of \cite{lagarias2011smooth}.

\begin{remark*}
Independently, using different methods, Gy\"ory has been able to show a similar result to that of these results, but over the base field $K$ \cite{gyorynewpaper}. We will discuss this further in section 7.

When this manuscript was completed, Professor Gy\"{o}ry informed me about a sharper $abc$ inequality over $\QQ$ and imaginary quadratic number fields by Mochizuki, Fesenko, Hoshi, Minamide and Porowski (submitted for publication). However, this result relies on Mochizuki's results in Inter-universal Teichm\"uller Theory \cite{mochizuki2021inter}, the veracity of which is currently being debated \cite{scholze2018abc}.
\end{remark*}

\section*{Acknowledgements}
I wish to thank Evgeniy Zorin for his advice and support on this paper, and to thank K\'alm\'an Gy\"ory for his helpful comments and encouragement on this paper. I would also like to thank Sanju Velani for his insightful comments on the paper.

\section{Preliminary Lemmas}
In this section we state pre-existing lemmas which we will repeatedly use throughout the proof of Theorem 1. In the rest of the paper, $\mathcal{C}_{11},\,\mathcal{C}_{12}, \dots$ denote effectively computable constants, and we will, where relevant, state what these constants depend on. In many cases, the constants depend on properties determined by a certain field; in these cases we will sometimes explicitly give which properties of the field the constants depend on.

We note that by Lemma 1 of \cite{bugeaud1996bounds}, we can find a system of fundamental units $\eta_{1},\dots,\, \eta_{r}$ where $r$ is the unit rank of $F$ such that
\begin{equation}\label{good system of fund units}
\prod_{i=1}^{r} \h{\eta_{i}} \leq \mathcal{C}_{11}R_{F}    
\end{equation}
where 
\[\mathcal{C}_{11} = \frac{\brac{\brac{r-1}!}^{2}}{2^{r-2}d^{r-1}}\]
and $R_{F}$ is the regulator of the field $F$. Throughout this paper we will use such a system of fundamental units, and often refer to them as "the fundamental units" of the field in question.

\begin{lemma}\label{pillars}
Let $F$ be a number field of degree $d$, and let $\alpha \in \Of{F}\backslash \Of{F}^{*}$. Then there is an effectively computable number $\mathcal{C}_{12}\brac{F}$, depending on the fundamental units of $\Of{F}$, and an $\epsilon \in \Of{F}^{*}$ such that
\[\house{\epsilon \alpha} \leq \mathcal{C}_{12} \abs{\Normf{F}{\alpha}}^{1/d}\]
where $\house{\alpha}$ denotes the house of $\alpha$. 
\end{lemma}
We recall that $\house{\alpha}$, the house of $\alpha$, is defined to be the maximal absolute value of the conjugates of $\alpha$ over $\CC$.
\begin{proof}
This is Lemma 1.3.8 from \cite{Natarajan2020}.
\end{proof}

\begin{lemma}\label{sunitbound}
Let $F$ be an algebraic number field of degree $d$ with set of normalised places $M_{F}$, and let $S$ be a finite subset of $M_{F}$ which contains $S_{\infty}$, the set of infinite places. Let $s$ be the cardinality of $S$, $\id{p}_{1},\,\dots,\, \id{p}_{t}$ the prime ideals corresponding to finite places of $S$ and let $P=\max_{i}\textrm{Nm}_{\QQ}^{F}\brac{\id{p}_{i}}$. Further let $R_{S}$ be the $S$-regulator of $F$. We note that $R_{S}=i_{S}R\prod_{i=1}^{t}\log \text{Nm}_{\QQ}^{F}\brac{\id{p}_{i}}$, where $i_{S}$ is a positive divisor of the class number $h_{F}$ of $F$ and $R$ is the regulator of $F$ (cf.  \cite{gyHory2008abc}\cite{gyory2010s}\cite{gyHory2006bounds}). Given $\alpha \comma \beta$, non-zero elements of $F$, we consider the $S$-unit equation $\alpha x + \beta y = 1$ in $x,\,y$, where $x \comma y$ are $S$-units. Let $r$ denote the unit rank of $F$ and let $\mathcal{R}=\max\curly{h_{F},\, \mathcal{C}_{13}(r,\,d)R}$, where $\mathcal{C}_{13}\brac{r,\,d}$ is given explicitly in \cite{gyHory2008abc}. Then, if $t=0$, all solutions $x,\, y$ of the above equation satisfy
\[\max\curly{h(x),\, h(y)}\leq \mathcal{C}_{13}\brac{r,\,d}\max\curly{h(\alpha),\, h(\beta),\, 1}.\]
If $t>0$ then we obtain
\[\max\curly{h(x),\, h(y)}\leq \mathcal{C}_{14}\brac{r,\,d}h_{F}R\brac{\log^{*}R}\mathcal{R}^{t+1}\brac{\log^{*}\mathcal{R}}\brac{\frac{P}{\log^{*}P}}R_{S}\max\curly{h(\alpha),\, h(\beta),\, 1},\]
where $\mathcal{C}_{14}\brac{r,\,d}$ is also explicitly given in \cite{gyHory2008abc}.
\end{lemma}
\begin{proof}
This is part of Theorem A from \cite{gyHory2008abc}.
\end{proof}

\begin{lemma}\label{Yu Varieties Bound}
Let $\alpha_{1} \comma \dots \comma \alpha_{n}$ be algebraic numbers and $F$ a number field of degree $d$ containing $\alpha_{1} \comma \dots \comma \alpha_{n}$. Let $\id{p}$ be a prime ideal of $\Of{F}$ lying above the rational prime $p$ with ramification index $e_{\id{p}}$ and residue class degree $f_{\id{p}}$. For $\alpha \in F \comma \alpha \neq 0$, we write $\ordid{\alpha}$ for the exponent to which $\id{p}$ divides the principal fractional ideal generated by $\alpha$ in $F$, and we set $\ordid{0}=+\infty$. Let $b_{1}\comma \dots b_{n}$ be integers, and set $\Theta = \alpha_{1}^{b_{1}} \cdots \alpha_{n}^{b_{n}}-1$. Assume that $\Theta \neq 0$. Finally, set $h'\brac{\alpha_{j}}=\max\curly{\h{\alpha_{j}} \comma \frac{1}{16e^{2}d^{2}}}$. Then
\begin{equation*}
    \ordid{\Theta}<\brac{16ed}^{2\brac{n+1}}n^{5/2} \log\brac{2nd}\log\brac{2d}\cdot e_{\id{p}}^{n}\frac{\emph{\text{Nm}}_{\QQ}^{F}\brac{\id{p}}}{\brac{\log \emph{\text{Nm}}_{\QQ}^{F}\brac{\id{p}}}^{2}}\prod_{i=1}^{n}h'\brac{\alpha_{i}}\log B,
\end{equation*}
where $B = \max \curly{\abs{b_{1}} \comma \dots \comma \abs{b_{n}}\comma 3}$.
\end{lemma}

\begin{proof}
This is a consequence of the main theorem of \cite{yu2007p}, given on page 190.
\end{proof}

\begin{lemma}\label{primeidealthm}
Let $F$ be a number field with ring of integers $\Of{F}$. Apply a total ordering to the prime ideals, so that if $\emph{\text{Nm}}_{\QQ}^{F}~\brac{\id{x}}>\emph{\text{Nm}}_{\QQ}^{F}\brac{\id{y}}$, then $\id{x} \succ \id{y}$. Arbitrarily order ideals of the same norm. Then there is an effectively computable positive constant $\mathcal{C}_{15}$ such that for every positive integer $r$ we have
\[\prod_{i=1}^{r}\frac{\emph{\text{Nm}}_{\QQ}^{F} ~\brac{\id{p}_{i}}}{\log \emph{\text{Nm}}_{\QQ}^{F}~\brac{\id{p}_{i}}}>\brac{\frac{r}{\mathcal{C}_{15}}}^{r}.\]
\end{lemma}
\begin{proof}
Let $\pi_{F}(x)$ denote the number of prime ideals in number field $F$ of norm less than or equal to $x$. By the Landau Prime Ideal Theorem \cite{landau1903neuer}, we know that $\pi_{F}(x)\sim \frac{x}{\log x}$. Partially order the prime ideals as in the statement of the Lemma. Then by Landau, \[\pi_{F} \brac{\text{Nm}_{\QQ}^{F} \brac{\id{p}_{j}}} \sim \frac{\text{Nm}_{\QQ}^{F} \brac{\id{p}_{j}}}{\log \brac{\text{Nm}_{\QQ}^{F}\brac{\id{p}_{j}}}}.\] Thus by Landau, there exists an effectively computable number $\mathcal{C}_{16}$ such that \[\frac{\text{Nm}_{\QQ}^{F}\brac{\id{p}_{j}}}{\log \brac{\text{Nm}_{\QQ}^{F}\brac{\id{p}_{j}}}} > j/\mathcal{C}_{16}.\] Thus, using the inequality $r!\geq \brac{r/e}^{r}$, we see that
\[\prod_{j=1}^{r} \frac{\text{Nm}_{\QQ}^{F}\brac{\id{p}_{j}}}{\log\brac{\text{Nm}_{\QQ}^{F}\brac{\id{p}_{j}}}}>\frac{r!}{\mathcal{C
}_{16}^{r}} \geq \brac{\frac{r}{\mathcal{C}_{17}e}}^{r} \geq \brac{\frac{r}{\mathcal{C}_{18}}}^{r},\]
which proves the claim in the lemma.
\end{proof}

We finally state a lemma that we will use to tidy our end arguments.

\begin{lemma}\label{last lemma}
If $x,\,a \in \mathbb{R}$ with $\frac{a}{\log a}<x$, then $a<\max \curly{e,\,2x \log x}$.
\end{lemma}
\begin{proof}
If $\frac{a}{\log a}<x$, then 
\begin{equation}\label{last lemma 1}
    a< x \log a.
\end{equation}

As $\frac{a}{\log a}<x$, we see that $\log a - \log \log a < \log x$. Further, we can show that $\frac{\log a}{2}<\log a - \log \log a$.

Combining these we get that 
\begin{equation}\label{last lemma 2}
    \log a< 2\log x.
\end{equation} 
Multiplying together \eqref{last lemma 1} and \eqref{last lemma 2} we obtain that $a< 2x \log x$. 
\end{proof}

\section{Proof of the Main Theorem}

Let $K$ be a number field with ring of integers $\Of{K}$, class number $h_{K}$, and Hilbert Class Field $L$. Let $\squar{K:\QQ}=d$, so by the tower law, $[L:\QQ]=h_{K}d$.

Take $a,\,b,\, c:=-a-b \in \Of{K}$ so that
\begin{equation}\label{defining feature of abc}
a+b+c=0,
\end{equation} 
with the assumption that $a\Of{K},\, b\Of{K}$ and $c\Of{K}$ are coprime. We write
\begin{align}
    a\Of{K}&=\id{p}_{1}^{e_{1}}\cdots \id{p}_{t}^{e_{t}} \nonumber \\
    b\Of{K}&=\id{q}_{1}^{f_{1}} \cdots \id{q}_{u}^{f_{u}} \nonumber \\
    c\Of{K}&=\id{r}_{1}^{g_{1}} \cdots \id{r}_{v}^{g_{v}},
\end{align}
where $\id{p}_{i},\, \id{q}_{j},\,$ and $\id{r}_{k}$ are prime ideals of $\Of{K}$ and $e_{i},\, f_{j},\, g_{k}$ are integers.

A key property of $L=HCF(K)$ is that every ideal $\id{I}$ of $\Of{K}$ is principal in $\Of{L}$; that is $\id{I}\Of{L} =\alpha \Of{L}$ for some $\alpha \in \Of{L}$. By Lemma \ref{pillars}, we can pick the generator of each ideal so that if $\alpha$ is the generator, then 
\begin{equation}\label{height and norm}
 \h{\alpha} \leq \log \house{\alpha}\leq \log \brac{\mathcal{C}_{19}\brac{K} \abs{\Normf{L}{\alpha}}^{1/d}}=\mathcal{C}_{20}\brac{K}\log \abs{\Normf{L}{\alpha}}.   
\end{equation}

We note that the dependence of the constants is on $K$ rather than $L$, as $L$ is uniquely determined by $K$. Further, for such algebraic $\alpha$ we have
\begin{align*}
\log \abs{\Normf{L}{\alpha}}\leq d\h{\alpha},
\end{align*}
giving us that 

\begin{equation}\label{height log comparison}
\mathcal{C}_{21}\brac{K}\log \abs{\Normf{L}{\alpha}}\leq \h{\alpha} \leq \mathcal{C}_{22}\brac{K}\log \abs{\Normf{L}{\alpha}}.
\end{equation}

Recalling this, we write
\begin{align}\label{primes in K, ideals in L}
	\id{p}_{i}\Of{L}&=a_{i}\Of{L} \nonumber \\
	\id{q}_{j}\Of{L}&=b_{j}\Of{L} \nonumber \\
	\id{r}_{k} \Of{L}&=c_{k}\Of{L}, 
\end{align}
where $a_{i}, \, b_{j}, \, c_{k}$ satisfy \eqref{height and norm}.

We can also write $a\Of{L},\, b\Of{L}$ and $c\Of{L}$ as a product of prime ideals of $\Of{L}$. Knowing this, we will write 
\begin{equation}\label{Definition of G}
G=\prod_{\substack{\id{P} \text{ prime ideal} \\ \id{P} \subset \Of{L} \\ \id{P} \mid \brac{abc}\Of{L} }} \text{Nm}_{\QQ}^{L} \brac{\id{P}}.
\end{equation}
Note that by our assumptions this is equivalent to taking the field to be $L$ in equation (1.7) of \cite{masser2002abc}. Further we will denote the prime ideal of $\Of{L}$ of largest norm dividing $a\Of{L}$ by $\id{p}_{a}$, and similarly for $b$ and $c$.

From \eqref{primes in K, ideals in L} we can write $a,\,b,\,c$ as follows:
\begin{align}\label{factorisation}
a&=u_{a}a_{1}^{e_{1}}\cdots a_{t}^{e_{t}} \nonumber \\
b&=u_{b}b_{1}^{f_{1}}\cdots b_{u}^{f_{u}} \nonumber \\
c&=u_{c}c_{1}^{g_{1}} \cdots c_{v}^{g_{v}}
\end{align}
where $u_{a},\, u_{b}$ and $u_{c}$ are units of $\mathcal{O}_{L}$. We will also often write $u_{a}a'+u_{b}b'+u_{c}c'=0$ where $a'=\prod_{i=1}^{s}a_{i}$ and similarly for $b'$ and $c'$. After relabeling if needed, we can assume that \eqref{double inequality} holds. We note that if $\h{c'}\leq 1$ then straightforwardly the claim holds. This is because necessarily $h(a')\leq h(b')\leq h(c')\leq 1$ and Theorem 1 readily follows from \eqref{relation between H and h}. Similarly, after some work we do below finding a bound on $h(c')$, we see that if $h(b')\leq 1$ then again we'll find the claim straightforwardly follows, and the same logic will hold if $h(a')\leq 1$. Thus we assume in the following that 
 \begin{equation}\label{1<ha<hb<hc}
 1<h(a')\leq h(b') \leq h(c').
 \end{equation}
 
  Dividing through by $u_{c}c'=c$ in \eqref{defining feature of abc} we obtain that
\begin{equation}\label{ready for s}
-\frac{u_{a}a'}{u_{c}c'}-\frac{u_{b}b'}{u_{c}c'}=1,
\end{equation}
so we are in a position to apply Lemma \ref{sunitbound}.

Before doing this, we note that by the remarks around \eqref{property of H}, $\log H_{L}\brac{a,\,b,\,c} = \log H_{L}\brac{\frac{a}{c},\, \frac{b}{c},\, 1}$. This allows us to move between representatives of the projective point $\squar{a:\, b:\, c} \in \mathbb{P}^{2}\brac{L}$.

First, following the notation from Lemma \ref{sunitbound}, initially take $S$ to be the set of infinite places of $L$. Applying this to \eqref{ready for s}, we obtain that
\begin{align}\label{first sunit bound}
\max\curly{\h{-\frac{u_{a}}{u_{c}}},\,\h{-\frac{u_{b}}{u_{c}}}}&=\max\curly{\h{\frac{u_{a}}{u_{c}}},\,\h{\frac{u_{b}}{u_{c}}}} \nonumber\\
&\leq\mathcal{C}_{23}\brac{K} \max \curly{\h{\frac{a'}{c'}},\, \h{\frac{b'}{c'}},\, 1}.    
\end{align}

We note similar bounds also hold if we divide \eqref{defining feature of abc} through by $u_{a}a'$ or $u_{b}b'$. Recall that $\h{xy}\leq \h{x}+\h{y}$ \cite{bombieri2007heights}\cite{Waldschmidt}. By this fact and our assumptions on $h(a'),\,h(b') \text{ and } h(c')$, we deduce that 
\begin{equation}
\max \curly{\h{\frac{a'}{c'}},\, \h{\frac{b'}{c'}},\, 1} \leq \max\curly{\h{a'}+\h{c'},\, \h{b'}+\h{c'}} \leq 2\h{c'}. 
\end{equation}

It thus follows that

\begin{align}\label{app 1 of first sunit bound}
\h{\frac{u_{a}a'}{u_{c}c'}}&\leq \h{\frac{u_{a}}{u_{c}}}+\h{a'}+\h{c'} \nonumber \\
&\leq \mathcal{C}_{24}\h{c'} +2\h{c'} \nonumber \\
&= \mathcal{C}_{25} \h{c'}. 
\end{align}

Similarly, we also obtain that
\begin{align}\label{app 2 of first sunit bound}
\h{\frac{u_{b}b'}{u_{c}c'}}&\leq \mathcal{C}_{26} \h{c'}.
\end{align}

We again apply Lemma \ref{sunitbound}, but this time with a different set of $S$-units. Now let $S=S_{\infty} \cup \curly{\id{p} \,:\, \id{p} \mid c\Of{L} }$ where $\id{p}$ refers both to the prime ideal $\id{p}$ and its corresponding place. Applying Lemma \ref{sunitbound} in this case to \eqref{ready for s}, we obtain that

\begin{align}\label{second sunit bound}
\max \curly{\h{\frac{u_{a}}{u_{c}c'}},\,\h{\frac{u_{b}}{u_{c}c'}} } &\leq \mathcal{C}_{27}\brac{K}\brac{\frac{\text{Nm}_{\QQ}^{L}\brac{\id{p}_{c}}}{\log^{*}\text{Nm}_{\QQ}^{L}\brac{\id{p}_{c}}}}\brac{\prod_{\substack{\id{p} \in \Of{L} \\ \id{p} \mid c\Of{L}}}\log \text{Nm}_{\QQ}^{L}\brac{\id{p}}}\max\curly{\h{a'},\, \h{b'},\, 1} \nonumber \\
&\leq \mathcal{C}_{27}\brac{K}\text{Nm}_{\QQ}^{L}\brac{\id{p}_{c}}\brac{\prod_{\substack{\id{p} \in \Of{L} \\ \id{p} \mid c\Of{L} \\ \id{p} \neq \id{p}_{c}}}\log \text{Nm}_{\QQ}^{L}\brac{\id{p}}}\h{b'}.
\end{align}

Again, $L$ is defined by $K$ uniquely, hence the dependence of the constants here on $K$ rather than $L$. Further, by \eqref{1<ha<hb<hc}, we have that $\max\curly{\h{a'},\, \h{b'},\, 1} = \h{b'}$.

From \eqref{second sunit bound} we obtain that
\begin{align}\label{app 1 of second sunit bound}
\h{\frac{u_{a}a'}{u_{c}c'}} &\leq \h{\frac{u_{a}}{u_{c}c'}}+\h{a'} \nonumber \\
&\leq \mathcal{C}_{27}\brac{K}\text{Nm}_{\QQ}^{L}\brac{\id{p}_{c}}\brac{\prod_{\substack{\id{p} \in \Of{L} \\ \id{p} \mid c\Of{L} \\ \id{p} \neq \id{p}_{c}}}\log \text{Nm}_{\QQ}^{L}\brac{\id{p}}} \h{b'} +\h{a'} \nonumber \\
&\leq \brac{\mathcal{C}_{27}\brac{K}\text{Nm}_{\QQ}^{L}\brac{\id{p}_{c}}\brac{\prod_{\substack{\id{p} \in \Of{L} \\ \id{p} \mid c\Of{L} \\ \id{p} \neq \id{p}_{c}}}\log \text{Nm}_{\QQ}^{L}\brac{\id{p}}}+1} \h{b'} \nonumber \\
&\leq \mathcal{C}_{28}\brac{K}\text{Nm}_{\QQ}^{L}\brac{\id{p}_{c}}\brac{\prod_{\substack{\id{p} \in \Of{L} \\ \id{p} \mid c\Of{L} \\ \id{p} \neq \id{p}_{c}}}\log \text{Nm}_{\QQ}^{L}\brac{\id{p}}} h(b').
\end{align}

Similarly, we find that
\begin{align}\label{app 2 of second sunit bound}
\h{\frac{u_{b}b'}{u_{c}c'}} &\leq \mathcal{C}_{29}\brac{K}\text{Nm}_{\QQ}^{L}\brac{\id{p}_{c}}\brac{\prod_{\substack{\id{p} \in \Of{L} \\ \id{p} \mid c\Of{L} \\ \id{p} \neq \id{p}_{c}}}\log \text{Nm}_{\QQ}^{L}\brac{\id{p}}} h(b').
\end{align}

 We now choose another set $S$, this time containing the infinite places and the finite places corresponding to the prime ideals dividing $bc \Of{L}$; that is, $S=S_{\infty} \cup \curly{\id{p}\,:\, \id{p} \mid bc \Of{L}}$. Applying Lemma \ref{sunitbound} to \eqref{ready for s} with this $S$
 we obtain that
 
 \begin{align}\label{third sunit bound}
\max \curly{\h{\frac{u_{a}}{u_{c}c'}},\,\h{\frac{u_{b}b'}{u_{c}c'}} } 
&\leq \mathcal{C}_{30}\brac{K}\max \curly{\text{Nm}_{\QQ}^{L}\brac{\id{p}_{b}},\,\text{Nm}_{\QQ}^{L}\brac{\id{p}_{c}} }\brac{\prod_{\substack{\id{p} \in \Of{L} \\ \id{p} \mid bc\Of{L} \\ }}\log \text{Nm}_{\QQ}^{L}\brac{\id{p}}}\max\curly{\h{a'},\, 1}.
\end{align}

From \eqref{1<ha<hb<hc} and \eqref{third sunit bound} we obtain that
\begin{align}\label{app 1 of third sunit bound}
\h{\frac{u_{a}a'}{u_{c}c'}} &\leq \h{\frac{u_{a}}{u_{c}c'}}+\h{a'} \nonumber \\
&\leq \mathcal{C}_{30}\brac{K}\max \curly{\text{Nm}_{\QQ}^{L}\brac{\id{p}_{b}},\,\text{Nm}_{\QQ}^{L}\brac{\id{p}_{c}} }\brac{\prod_{\substack{\id{p} \in \Of{L} \\ \id{p} \mid bc\Of{L} \\ }}\log \text{Nm}_{\QQ}^{L}\brac{\id{p}}}\h{a'} +\h{a'} \nonumber \\
&\leq \brac{\mathcal{C}_{30}\brac{K}\max \curly{\text{Nm}_{\QQ}^{L}\brac{\id{p}_{b}},\,\text{Nm}_{\QQ}^{L}\brac{\id{p}_{c}} }\brac{\prod_{\substack{\id{p} \in \Of{L} \\ \id{p} \mid bc\Of{L} \\ }}\log \text{Nm}_{\QQ}^{L}\brac{\id{p}}}+1} \h{a'} \nonumber \\
&\leq \brac{\mathcal{C}_{31}\brac{K}\max \curly{\text{Nm}_{\QQ}^{L}\brac{\id{p}_{b}},\,\text{Nm}_{\QQ}^{L}\brac{\id{p}_{c}} }\brac{\prod_{\substack{\id{p} \in \Of{L} \\ \id{p} \mid bc\Of{L} \\ }}\log \text{Nm}_{\QQ}^{L}\brac{\id{p}}}} \h{a'}.
\end{align}

Similarly, we find that
\begin{align}\label{app 2 of third sunit bound}
\h{\frac{u_{b}b'}{u_{c}c'}} &\leq \mathcal{C}_{32}\max \curly{\text{Nm}_{\QQ}^{L}\brac{\id{p}_{b}},\,\text{Nm}_{\QQ}^{L}\brac{\id{p}_{c}} }\brac{\prod_{\substack{\id{p} \in \Of{L} \\ \id{p} \mid bc\Of{L} }}\log \text{Nm}_{\QQ}^{L}\brac{\id{p}}} h(a').
\end{align}

By consideration of \eqref{app 1 of first sunit bound}, \eqref{app 2 of first sunit bound}, we see that
\begin{align}\label{D}
\max \curly{\h{\frac{u_{a}a'}{u_{c}c'}},\,\h{\frac{u_{b}b'}{u_{c}c'}}}&< \max\curly{\mathcal{C}_{25},\, \mathcal{C}_{26}}\h{c'} \nonumber \\
&=\mathcal{C}_{33} \h{c'},
\end{align}

while \eqref{app 1 of second sunit bound} and \eqref{app 2 of second sunit bound} show that

\begin{align}\label{E}
\max \curly{\h{\frac{u_{a}a'}{u_{c}c'}},\,\h{\frac{u_{b}b'}{u_{c}c'}}}&< \max\curly{\mathcal{C}_{28}, \mathcal{C}_{29}}\text{Nm}_{\QQ}^{L}\brac{\id{p}_{c}}\brac{\prod_{\substack{\id{p} \in \Of{L} \\ \id{p} \mid c\Of{L} \\ \id{p} \neq \id{p}_{c}}}\log \text{Nm}_{\QQ}^{L}\brac{\id{p}}} h(b') \nonumber
\\
&= \mathcal{C}_{34} \text{Nm}_{\QQ}^{L}\brac{\id{p}_{c}}\brac{\prod_{\substack{\id{p} \in \Of{L} \\ \id{p} \mid c\Of{L} \\ \id{p} \neq \id{p}_{c}}}\log \text{Nm}_{\QQ}^{L}\brac{\id{p}}} h(b').
\end{align}

In the same way, it follows from \eqref{app 1 of third sunit bound} and \eqref{app 2 of third sunit bound} that
\begin{align}\label{F}
\max \curly{\h{\frac{u_{a}a'}{u_{c}c'}},\,\h{\frac{u_{b}b'}{u_{c}c'}}}&< \max\curly{\mathcal{C}_{31}, \mathcal{C}_{32}}\max \curly{\text{Nm}_{\QQ}^{L}\brac{\id{p}_{b}},\,\text{Nm}_{\QQ}^{L}\brac{\id{p}_{c}} }\brac{\prod_{\substack{\id{p} \in \Of{L} \\ \id{p} \mid bc\Of{L} }}\log \text{Nm}_{\QQ}^{L}\brac{\id{p}}} h(a') \nonumber
\\
&= \mathcal{C}_{35} \max \curly{\text{Nm}_{\QQ}^{L}\brac{\id{p}_{b}},\,\text{Nm}_{\QQ}^{L}\brac{\id{p}_{c}} }\brac{\prod_{\substack{\id{p} \in \Of{L} \\ \id{p} \mid bc\Of{L}}}\log \text{Nm}_{\QQ}^{L}\brac{\id{p}}} h(a').
\end{align}

We next prove the following lemma.

\begin{lemma}\label{relation between ord and h}
Let $\alpha \in \curly{a,\,b,\,c}$. Then
\begin{equation}
\h{\alpha'} \leq \mathcal{C}_{36}\brac{\max_{\id{p} \mid \prin{\alpha}_{L}} \ordid{\alpha}} \log G.
\end{equation}
\end{lemma}
\begin{proof}
If $\alpha$ is a unit, then we define $\max \,\ordid{\alpha}:=1$. 

By construction of $\alpha'$, we have that $\h{\alpha'} \leq \mathcal{C}_{37}\log \text{Nm}_{\QQ}^{L}\brac{\alpha\Of{L}}$. This follows from the fact that $$\abs{\Normf{L}{\alpha'}}=\abs{\Normf{L}{\alpha}}=\text{Nm}_{\QQ}^{L}\brac{\alpha\Of{L}}.$$

We write a factorisation of $\alpha\Of{L} = \id{P}_{1,L}^{g_{1,L}} \cdots \id{P}_{u_{L},L}^{g_{u_{L},L}}$ into prime ideals of $\Of{L}$. Note this may be different to the ideals given in \eqref{primes in K, ideals in L}, as the ideals in \eqref{primes in K, ideals in L} may not be prime. Working with this prime factorisation, we obtain that
\begin{align*}
\log \text{Nm}_{\QQ}^{L}\brac{\alpha\Of{L}} &= \log \brac{\prod_{i=1}^{u_{L}}\text{Nm}_{\QQ}^{L}\brac{\id{p}_{i,L}^{u_{i},L}}} \\
&= \sum_{i=1}^{u_{L}}u_{i,L}\log \brac{\text{Nm}_{\QQ}^{L}\brac{\id{p}_{i,L}}} \\
&\leq \brac{\max_{\id{p} \mid \prin{\alpha}_{L}} \ordid{\alpha}} \log G.
\end{align*}
The claim then follows. 
\end{proof}

 It follows immediately from \eqref{relation between H and h}, \eqref{D} and Lemma \ref{relation between ord and h} that
\begin{equation}\label{H < ord c}
\frac{\log H_{L}\brac{a,\,b,\,c}}{\mathcal{C}_{33}\log G} \leq \max_{\id{p} \mid \prin{c}_{L}} \ordid{c}.
\end{equation}

Similarly, it follows from \eqref{relation between H and h}, \eqref{E} and Lemma \ref{relation between ord and h} that
\begin{equation}\label{H < ord b}
\frac{\log H_{L}\brac{a,\,b,\,c}}{\mathcal{C}_{34}\log G\, \text{Nm}_{\QQ}^{L}\brac{\id{p}_{c}}\brac{\prod_{\substack{\id{p} \in \Of{L} \\ \id{p} \mid c\Of{L} \\ \id{p} \neq \id{p}_{c}}}\log \text{Nm}_{\QQ}^{L}\brac{\id{p}}}} \leq \max_{\id{p} \mid \prin{b}_{L}} \ordid{b}.
\end{equation}

Further, it follows from \eqref{relation between H and h}, \eqref{F} and Lemma \ref{relation between ord and h} that
\begin{equation}\label{H < ord a}
\frac{\log H_{L}\brac{a,\,b,\,c}}{\mathcal{C}_{35}\log G\, \max \curly{\text{Nm}_{\QQ}^{L}\brac{\id{p}_{b}},\,\text{Nm}_{\QQ}^{L}\brac{\id{p}_{c}} }\brac{\prod_{\substack{\id{p} \in \Of{L} \\ \id{p} \mid bc\Of{L} }}\log \text{Nm}_{\QQ}^{L}\brac{\id{p}}}} \leq \max_{\id{p} \mid \prin{b}_{L}} \ordid{a}.
\end{equation}

We will use Lemma \ref{Yu Varieties Bound} to establish upper bounds for the right-hand sides of \eqref{H < ord c}, \eqref{H < ord b} and \eqref{H < ord a}. In order to do this we need to write $\ordid{c}$, $\ordid{b}$ and $\ordid{a}$ in a form where we're able to use Lemma \ref{Yu Varieties Bound}.

By the coprimeness of $a\Of{L},\, b\Of{L}$ and $c \Of{L}$ we see that
\begin{align}\label{ordid c 1}
\ordid{c}&=\ordid{\frac{c}{b}} \nonumber \\
&=\ordid{\frac{-a-b}{b}} \nonumber \\
&=\ordid{-\frac{a}{b}-1} \nonumber \\
&= \ordid{- \frac{u_{a}}{u_{b}}a_{1}^{e_{1}}\cdots a_{t}^{e_{t}}b_{1}^{f_{1}}\dots b_{u}^{f_{u}}-1}. 
\end{align}

Similarly, 
\begin{align}\label{ordid b 1}
\ordid{b}&=\ordid{\frac{b}{a}} \nonumber \\
&=\ordid{\frac{-a-c}{a}} \nonumber \\
&=\ordid{-\frac{c}{a}-1} \nonumber \\
&= \ordid{- \frac{u_{c}}{u_{a}}c_{1}^{g_{1}}\cdots c_{v}^{g_{v}}a_{1}^{-e_{1}}\dots a_{t}^{-e_{t}}-1}, 
\end{align}

and

\begin{align}\label{ordid a 1}
\ordid{a}&=\ordid{\frac{a}{c}} \nonumber \\
&=\ordid{\frac{-b-c}{c}} \nonumber \\
&=\ordid{-\frac{b}{c}-1} \nonumber \\
&= \ordid{- \frac{u_{b}}{u_{c}}b_{1}^{f_{1}}\cdots b_{u}^{f_{u}}c_{1}^{-g_{1}}\dots c_{v}^{-g_{v}}-1}.
\end{align}

To apply Lemma \ref{Yu Varieties Bound} we need to bound exponents $e_{i}, \, f_{j},\, g_{k}$ and bound the heights of the units $\h{u_{a}},\, \h{u_{b}}$ and $\h{u_{c}}$.

First note that
\begin{align}\label{exponent bound}
\max \curly{\ordid{a},\, \ordid{b},\, \ordid{c}}\leq \log H_{L}\brac{a,\,b,\,c},
\end{align}
directly from the definition of $H_{L}$.

This follows from the definitions of projective and absolute logarithmic heights.

In order to use Yu's bound, we need to manage the heights of $u_{a},\, u_{b}$ and $u_{c}$. To do this, we will use fundamental units of $\Of{L}^{*}$. By Dirichlet's Unit Theorem, there exist fundamental units $\xi_{1}, \dots, \xi_{r}$ of $\Of{L}$, where $r$ is the unit rank of $\Of{L}$ such that all units $u$ of $\Of{L}$ can be written $u=\mu \xi_{1}^{\delta_{1}} \cdots \xi_{r}^{\delta_{r}}$, with $\mu$ a root of unity. We note again that finding a set of fundamental units is computable, for example see \cite{buchmann1987computation}, and we can find a nice system satisfying \eqref{good system of fund units}. Indeed, we could use any system of fundamental units, but this choice is helpful should one wish to explicitly compute the constants. Thus once found, the product $\prod_{i=1}^{r}h'(\xi_{i})$ we will obtain applying Lemma \ref{Yu Varieties Bound} can be upper bounded constants depending on the field. It remains to find an upper bound for $\max_{i} \delta_{i}$.

Note that \eqref{first sunit bound} gives us that 
\[\h{\frac{u_{a}}{u_{c}}} \leq \mathcal{C}_{38}\h{c'}.\] 
Further, 
\[\h{c'} \leq \mathcal{C}_{39}\brac{K}\log \abs{\Normf{L}{c}} \leq \mathcal{C}_{40}\brac{K} \h{c} \leq \mathcal{C}_{41}\brac{K}\log H_{L}(a,\,b,\, c).\]
 It follows from the above comments that 
 \begin{equation}\label{Bound to give delta bound}
\h{\frac{u_{a}}{u_{c}}}\leq \mathcal{C}_{42}\brac{K}\log H_{L}(a,\,b,\,c). 
\end{equation}

Let $L$ have $r_{1}$ real embeddings $\epsilon_{1},\dots,\,\epsilon_{r_{1}}$ and $2r_{2}$ complex embeddings $\epsilon_{r_{1}+1}, \, \overline{\epsilon_{r_{1}+1}},\dots, \epsilon_{r_{2}},\, \overline{\epsilon_{r_{2}}}$. By Dirichlet's Unit Theorem, there are $r:=r_{1}+r_{2}-1$ fundamental units $\xi_{1}, \dots, \xi_{r}$ such that for any unit $u \in \mathcal{O}_{L}^{*}$, $u=\mu \xi_{1}^{\delta_{1}} \cdots \xi_{r}^{\delta_{r}} $ where $\mu$ is a root of unity in $L$ and $\delta_{i} \in \mathbb{Z}$ for all $i$.

We now prove a lemma that gives an upper bound for $\max_{i} \abs{\delta_{i}}$.

\begin{lemma}\label{upper bound delta}
Given the set up above, 
\begin{equation}\label{delta bound}
    \max_{i} \abs{\delta_{i}}\leq \mathcal{C}_{43}\brac{K}\log H_{L}(a,\,b,\,c).
\end{equation}
\end{lemma}

\begin{proof}

 Recall that there are $r+1$ distinct embeddings of $L$ into $\mathbb{C}$. Let us denote these embeddings by $e_i$, $i=1\dots,r+1$ . Furthermore for all $i=1\dots,r+1$, let us define
\[
\begin{aligned}
\varepsilon_i: L &\rightarrow \mathbb{R}\\
\alpha &\to \log|e_i(\alpha)|.
\end{aligned}
\]
Let $u$ be an element of $L$. We see that
\[
h\brac{u}=\frac12\sum_{i=1}^{r+1}N_i\abs{\varepsilon_i\brac{u}},
\]
where $N_i=1$ if the image of $e_i$ is a subset of $\R$, and $N_i=2$ in the complimentary case, when the image of $e_i$ is not a subset of $\R$. This is because
\begin{align*}
    h(u)=\sum_{i=1}^{r+1}N_{i}\log \max\brac{\abs{e_{i}(u),\, 1}}=-\sum_{i=1}^{r+1}N_{i}\log \min\brac{\abs{e_{i}(u),\, 1}},
\end{align*}
where $N_{i}$ is defined as above. Thus 
\begin{align*}
    2h(u)&=\sum_{i=1}^{r+1}N_{i}\log \max\brac{\abs{e_{i}(u),\, 1}}-\sum_{i=1}^{r+1}N_{i}\log \min\brac{\abs{e_{i}(u),\, 1}}\\
    &= \sum_{i=1}^{r+1}N_{i} \abs{\log \abs{e_{i}(u)}}\\
    &=\sum_{i=1}^{r+1}N_{i} \abs{\varepsilon_{i}(u)},
\end{align*}
and thus the identity follows.

We now take advantage of some properties of $u_{a},\,u_{b}$ and $u_{c}$ so we will write them explicitly. In what follows we use the case $u= \frac{u_{a}}{u_{c}}$, but it is true for the other relevant quotients of $u_{a},\, u_{b}$ and $u_{c}$. Write $\frac{u_{a}}{u_{c}}=\mu \xi_{1}^{\delta_{1}}\cdots \xi_{r}^{\delta_{r}}$ where $\mu$ is a root of unity. From the comments above it follows that 
\begin{equation} \label{dev_h}
\begin{aligned}
h\brac{\frac{u_{a}}{u_{c}}}&=\frac12\sum_{i=1}^{r+1}N_i\left|\varepsilon_i\left(\mu \prod_{j=1}^r\xi_j^{\delta_j} \right)\right|\\
&=\frac12\sum_{i=1}^{r+1}N_i\left| \sum_{j=1}^r\delta_j\varepsilon_i(\xi_j)\right|,
\end{aligned}
\end{equation}
where we lose the $\mu$ as it is a root of unity, so for all $i$, $\varepsilon_{i}\brac{\mu}=0$.

From \eqref{Bound to give delta bound}, we know that $\h{\frac{u_{a}}{u_{c}}}\leq \mathcal{C}_{42}\log H_{L}(a,\,b,\,c)$, giving us an upper bound for the absolute logarithmic height of the unit. Along with~\eqref{dev_h}, this implies that, for all $i=1,\dots,r+1$, we have
\begin{equation} \label{matrix_ub_K}
\left| \sum_{j=1}^r\delta_j\varepsilon_i(\xi_j)\right|<\mathcal{C}_{44}\brac{K} \log H_{L}(a,\,b,\,c).
\end{equation}
That is, \eqref{Bound to give delta bound} implies that all the exponents $\delta_j$, $j=1,\dots,r$ satisfy~\eqref{matrix_ub_K}. This holds for all $\varepsilon_{i},\, i=1,\dots,\, r+1$, so \eqref{matrix_ub_K} gives us a system of $r+1$ inequalities.

We pick any $r$ inequalities of $r+1$ in the system~\eqref{matrix_ub_K}. For the sake of concreteness, let us take the first $r$ inequalities. We are going to deduce the upper bound for the system of inequalities~\eqref{matrix_ub_K} where $i=1,\dots,r$. Note that the left-hand side of these inequalities are coordinates of the vector
\[
M\begin{pmatrix} \delta_1 \\ \vdots \\ \delta_r \end{pmatrix},
\]
where the matrix $M$ is defined by
\[
M:=(\varepsilon_i(\xi_j))_{1\leq i,j\leq r}.
\]
By definition, the absolute value of determinant of $M$ is equal to the regulator of the number field $L$, and is thus non-zero. Hence $M$ is non-degenerate. Importantly, the matrix $M$ depends only on the number field $L$. Further, as the value of the regulator is independent of the choice of the $r$ inequalities we picked, it shows that our choice of inequalities is irrelevant and we obtain the same result given a different choice of $r$ inequalities from the $r+1$ in \eqref{matrix_ub_K}.

It follows that the solutions to the system of inequalities~\eqref{matrix_ub_K} for $i=1,\dots,r$ are given by $M^{-1}\mathcal{B}$, where $\mathcal{B}$ is an $r$-dimensional cube $[-\mathcal{C}_{44}\brac{K}\log H_{L}\brac{a,\,b,\,c},\mathcal{C}_{44}\brac{K}\log H_{L}\brac{a,\,b,\,c}]^r$. Thus these solutions form a parallepiped, the form of which depends on $M$ (hence eventually on $L$ only, which is uniquely determined by $K$) and the linear size is given by $\mathcal{C}_{45}\brac{K}\log H_{L}(a,\,b,\,c)$. This means that the solutions $\delta_i$ have an upper bound of the form $\mathcal{C}_{46}\mathcal{C}_{45} \log H(a,\,b,\,c)$, where the constant $\mathcal{C}_{46}$ depends on $M$ (hence actually depends on $K$) only. We thus conclude that 
\begin{equation*}
    \max_{i} \abs{\delta_{i}}\leq \mathcal{C}_{47}\brac{K}\log H_{L}(a,\,b,\,c),
\end{equation*}
as claimed.
\end{proof}

Importantly, as commented during the proof, the method is not changed if we choose a different unit such as $\frac{u_{a}}{u_{b}}$ and so on. Thus this lemma holds for all relevant units in this paper.

We return to considering \eqref{ordid c 1}. We can now use the above after writing the unit in terms of fundamental units as follows:

\begin{align}\label{ordid c}
\ordid{c} &= \ordid{- \frac{u_{a}}{u_{b}}a_{1}^{e_{1}}\cdots a_{t}^{e_{t}}b_{1}^{f_{1}}\dots b_{u}^{f_{u}}-1} \nonumber \\
&=\ordid{\mu \xi_{1}^{\delta_{1}}\cdots \xi_{r}^{\delta_{r}}a_{1}^{e_{1}}\cdots a_{t}^{e_{t}}b_{1}^{f_{1}}\cdots b_{u}^{f_{u}}-1}
\end{align}
where $\mu$ is a root of unity. We are in a position to apply Lemma \ref{Yu Varieties Bound}. Using the notation of Lemma \ref{Yu Varieties Bound}, from \eqref{exponent bound} and Lemma \ref{upper bound delta} we obtain that $\log B \leq \mathcal{C}_{42}\brac{K} \log \log H_{L}(a,\,b,\,c)$.

Applying Lemma \ref{Yu Varieties Bound} on \eqref{ordid c}, we obtain that
\begin{align}
\ordid{c} < & \mathcal{C}_{48}\brac{K}^{r+t+u+2}\brac{r+t+u+1}^{5/2}\log\brac{2d\brac{r+t+u+1}}\frac{\text{Nm}_{\QQ}^{L}\brac{\id{p}_{c}}}{\brac{\log \text{Nm}_{\QQ}^{L}\brac{\id{p}_{c}}}^{2}} \nonumber \\
& h'(\mu)h'(\xi_{1})\cdots h'(\xi_{r}) h'(a_{1}) \cdots h'(a_{t}) h'(b_{1}) \cdots h'(b_{u}) \log \log H_{L}(a,\,b,\,c). 
\end{align}
Note that $\h{\mu}=0$ as $\mu$ is a root of unity so $h'(\mu)= \frac{1}{16e^{2}d^{2}}$, which we take into the constant. Further, we recall our system of fundamental units satisfies \eqref{good system of fund units} so we take $\prod_{i=1}^{r}h'(\xi_{i})$ into the constant.

We further note that $\abs{\Normf{L}{a_{i}}}= \text{Nm}_{\QQ}^{L} \brac{\prin{a_{i}}}=\brac{\text{Nm}_{\QQ}^{K} \brac{\id{p}_{i}}}^{f_{K}}$, and similarly for $b_{j}$ and $c_{k}$ \cite{lang2013algebraic}. We recall that by definition, $f_{K} \leq d$  \cite{neukirch2013algebraic}. Further, the norms of all these prime ideals are greater than $1$, so for all $x\in \curly{a_{1},\dots,\,a_{s},\, b_{1},\dots,\, b_{t}},$ if $\id{a}_{x}$ is the prime ideal associated with $x$, $h'(x)\leq \mathcal{C}_{49}\brac{K} \log \text{Nm}_{\QQ}^{K}\brac{\id{a}_{x}} $. Putting this together with the inequality above, with other bounds used as necessary, we obtain that
\begin{align}\label{ordid(c) leq bound}
\ordid{c}\leq &\mathcal{C}_{50}\brac{K}^{t+u}\brac{r+t+u+1}^{7/2} \log \log H_{L}(a,\,b,\,c)\nonumber\\ &\text{Nm}_{\QQ}^{L}\brac{\id{p}_{c}} \prod_{i=1}^{t} \log \brac{\text{Nm}_{\QQ}^{K}\brac{\id{p}_{i}}}\cdot \prod_{j=1}^{u} \log \brac{\text{Nm}_{\QQ}^{K}\brac{\id{q}_{j}}}.
\end{align}

Similarly, we see that by considering \eqref{ordid b 1} in the same way as above, we obtain that
\begin{align}\label{ordid b}
\ordid{b}&=\ordid{- \frac{u_{c}}{u_{a}}c_{1}^{g_{1}}\cdots c_{v}^{g_{v}}a_{1}^{-e_{1}}\dots a_{t}^{-e_{t}}-1} \nonumber \\
&=\ordid{\mu' \xi_{1}^{\delta'_{1}}\cdots \xi_{r}^{\delta'_{r}}c_{1}^{g_{1}}\cdots c_{v}^{g_{v}}a_{1}^{-e_{1}}\dots a_{t}^{-e_{t}}-1}.
\end{align}
Following the same line of reasoning as above we obtain that
\begin{align}\label{ordid(b) leq bound}
\ordid{b}\leq& \mathcal{C}_{51}\brac{K}^{t+v}\brac{r+t+v+1}^{7/2} \log \log H_{L}(a,\,b,\,c)\nonumber \\ &\text{Nm}_{\QQ}^{L}\brac{\id{p}_{b}} \prod_{i=1}^{t} \log \brac{\text{Nm}_{\QQ}^{K}\brac{\id{p}_{i}}}\cdot \prod_{j=1}^{v} \log \brac{\text{Nm}_{\QQ}^{K}\brac{\id{r}_{j}}}.
\end{align}

In the same way, by considering \eqref{ordid a 1} we further obtain that
\begin{align}\label{ordid a}
\ordid{a}&=\ordid{- \frac{u_{b}}{u_{c}}b_{1}^{f_{1}}\cdots b_{u}^{f_{u}}c_{1}^{-g_{1}}\dots c_{v}^{-g_{v}}-1} \nonumber \\
&=\ordid{\mu'' \xi_{1}^{\delta'_{1}}\cdots \xi_{r}^{\delta'_{r}}b_{1}^{f_{1}}\cdots b_{u}^{f_{u}}c_{1}^{-g_{1}}\dots c_{v}^{-g_{v}}-1},
\end{align}

and as before, applying Lemma \ref{Yu Varieties Bound} gives us that
\begin{align}\label{ordid(a) leq bound}
\ordid{a}\leq& \mathcal{C}_{52}\brac{K}^{u+v}\brac{r+u+v+1}^{7/2} \log \log H_{L}(a,\,b,\,c)\nonumber \\ &\text{Nm}_{\QQ}^{L}\brac{\id{p}_{a}} \prod_{i=1}^{u} \log \brac{\text{Nm}_{\QQ}^{K}\brac{\id{q}_{i}}}\cdot \prod_{j=1}^{v} \log \brac{\text{Nm}_{\QQ}^{K}\brac{\id{r}_{j}}}.
\end{align}

From this point, all constants depend on the field $K$, in particular on the degree of the field $d$, so we omit these dependencies. 

By combining \eqref{H < ord c} and \eqref{ordid(c) leq bound} we obtain that
\begin{align}\label{H< first bound}
\frac{\log H_{L}(a,\,b,\,c)}{\log \log H_{L}(a,\,b,\,c)}<&\mathcal{C}_{53}^{t+u}\brac{r+t+u+1}^{7/2}\log G \text{Nm}_{\QQ}^{L}\brac{\id{p}_{c}}\nonumber \\
 &\prod_{i=1}^{t} \log \brac{\text{Nm}_{\QQ}^{K}\brac{\id{p}_{i}}}\cdot \prod_{j=1}^{u} \log \brac{\text{Nm}_{\QQ}^{K}\brac{\id{q}_{j}}}.
\end{align}

Similarly, combining \eqref{H < ord b} and \eqref{ordid(b) leq bound} gives us that
\begin{align}\label{H< second bound}
\frac{\log H_{L}(a,\,b,\,c)}{\log \log H_{L}(a,\,b,\,c)}<&\mathcal{C}_{54}^{t+v}\brac{r+t+v+1}^{7/2}\log G \cdot \text{Nm}_{\QQ}^{L}\brac{\id{p}_{b}}\cdot\text{Nm}_{\QQ}^{L}\brac{\id{p}_{c}} \nonumber \\ &\prod_{i=1}^{t} \log \brac{\text{Nm}_{\QQ}^{K}\brac{\id{p}_{i}}}\cdot \prod_{j=1}^{v} \log \brac{\text{Nm}_{\QQ}^{K}\brac{\id{r}_{j}}}\cdot \prod_{\substack{\id{p} \in \Of{L} \\ \id{p} \mid c\Of{L} \\ \id{p} \neq \id{p}_{c}}}\log \text{Nm}_{\QQ}^{L}\brac{\id{p}}.
\end{align}

Applying the same idea, combining \eqref{H < ord a} and \eqref{ordid(a) leq bound} gives us that
\begin{align}\label{H< third bound}
\frac{\log H_{L}(a,\,b,\,c)}{\log \log H_{L}(a,\,b,\,c)}<&\mathcal{C}_{55}^{u+v}\brac{r+u+v+1}^{7/2}\log G \cdot \text{Nm}_{\QQ}^{L}\brac{\id{p}_{a}}\cdot\max \curly{\text{Nm}_{\QQ}^{L}\brac{\id{p}_{b}},\,\text{Nm}_{\QQ}^{L}\brac{\id{p}_{c}} } \nonumber \\ &\prod_{i=1}^{u} \log \brac{\text{Nm}_{\QQ}^{K}\brac{\id{p}_{i}}}\cdot \prod_{j=1}^{v} \log \brac{\text{Nm}_{\QQ}^{K}\brac{\id{r}_{j}}}\cdot \prod_{\substack{\id{p} \in \Of{L} \\ \id{p} \mid bc\Of{L} }}\log \text{Nm}_{\QQ}^{L}\brac{\id{p}}.
\end{align}

Multiplying together \eqref{H< first bound},\, \eqref{H< second bound} and \eqref{H< third bound} and bounding some terms for ease, we obtain that

\begin{align}\label{multiplied together 2}
\brac{\frac{\log H_{L}(a,\,b,\,c)}{\log \log H_{L}(a,\,b,\,c)}}^{3} < &\mathcal{C}_{56}^{t+u+v} \brac{r+t+u+v}^{21/2} \brac{\log G}^{3}\nonumber \\
&\brac{\text{Nm}_{\QQ}^{L}\brac{\id{p}_{a}} \text{Nm}_{\QQ}^{L}\brac{\id{p}_{b}} \text{Nm}_{\QQ}^{L}\brac{\id{p}_{c}}^{2}\max \curly{\text{Nm}_{\QQ}^{L}\brac{\id{p}_{b}},\,\text{Nm}_{\QQ}^{L}\brac{\id{p}_{c}} }} \nonumber \\
&\brac{ \prod_{i=1}^{t}  \log \brac{\text{Nm}_{\QQ}^{K}\brac{\id{p}_{i}}} \cdot \prod_{j=1}^{u} \log \brac{\text{Nm}_{\QQ}^{K}\brac{\id{q}_{j}}} \cdot\prod_{k=1}^{v} \log \brac{\text{Nm}_{\QQ}^{K}\brac{\id{r}_{k}}}}^{3} \nonumber \\ & \prod_{\substack{\id{p} \in \Of{L} \\ \id{p} \mid c\Of{L} \\ \id{p} \neq \id{p}_{c}}}\log \text{Nm}_{\QQ}^{L}\brac{\id{p}}\cdot \prod_{\substack{\id{p} \in \Of{L} \\ \id{p} \mid bc\Of{L} }}\log \text{Nm}_{\QQ}^{L}\brac{\id{p}}.
\end{align}

We note that $r$ depends on the field, so we can write $\brac{r+\brac{t+u+v}}^{21/2} \leq \mathcal{C}_{57}\brac{t+u+v}^{21/2}$. Further, for sufficiently large $\mathcal{C}_{57}$ this will absorb $\brac{t+u+v}^{21/2}$, so we can move this into the constant. We thus obtain that

\begin{align}\label{multiplied together (r+s+t) dealt with}
\brac{\frac{\log H_{L}(a,\,b,\,c)}{\log \log H_{L}(a,\,b,\,c)}}^{3} < &\mathcal{C}_{58}^{t+u+v} \brac{\log G}^{3} \brac{\text{Nm}_{\QQ}^{L}\brac{\id{p}_{a}} \text{Nm}_{\QQ}^{L}\brac{\id{p}_{b}} \text{Nm}_{\QQ}^{L}\brac{\id{p}_{c}}^{2}\max \curly{\text{Nm}_{\QQ}^{L}\brac{\id{p}_{b}},\,\text{Nm}_{\QQ}^{L}\brac{\id{p}_{c}} }} \nonumber \\
&\brac{ \prod_{i=1}^{t}  \log \brac{\text{Nm}_{\QQ}^{K}\brac{\id{p}_{i}}} \cdot \prod_{j=1}^{u} \log \brac{\text{Nm}_{\QQ}^{K}\brac{\id{q}_{j}}} \cdot\prod_{k=1}^{v} \log \brac{\text{Nm}_{\QQ}^{K}\brac{\id{r}_{k}}}}^{3} \nonumber \\ & \prod_{\substack{\id{p} \in \Of{L} \\ \id{p} \mid c\Of{L} \\ \id{p} \neq \id{p}_{c}}}\log \text{Nm}_{\QQ}^{L}\brac{\id{p}}\cdot \prod_{\substack{\id{p} \in \Of{L} \\ \id{p} \mid bc\Of{L} }}\log \text{Nm}_{\QQ}^{L}\brac{\id{p}}.
\end{align}

Next we aim to deal with $\prod_{i=1}^{t}  \log \brac{\text{Nm}_{\QQ}^{K}\brac{\id{p}_{i}}} \cdot \prod_{j=1}^{u} \log \brac{\text{Nm}_{\QQ}^{K}\brac{\id{q}_{j}}} \cdot\prod_{k=1}^{v} \log \brac{\text{Nm}_{\QQ}^{K}\brac{\id{r}_{k}}}$. First note that $\text{Nm}_{\QQ}^{K} \brac{\id{P}}^{h_{K}} = \text{Nm}_{\QQ}^{L} \brac{\id{P}\Of{L}}$, where $\id{P}$ is a prime ideal of $\Of{K}$. We follow an idea from the first part of Section 3 of \cite{stewart2001abc}. Let $N$ be the number of prime ideals of $\Of{L}$ such that the prime ideal $\id{P} \mid \brac{abc}\Of{L}$. By definition, these all lie above primes $\id{p}$ of $\Of{K}$, so $N\geq t+u+v$. Thus from these comments and Lemma \ref{primeidealthm} we obtain that
\begin{align}\label{prelims for log primes}
\brac{\frac{t+u+v}{\mathcal{C}_{59}}}^{t+u+v} \leq \brac{\frac{N}{\mathcal{C}_{60}}}^{N} < G,
\end{align}
where $\mathcal{C}_{60}$ is the constant given by Lemma \ref{primeidealthm}.
It follows that
\begin{equation}\label{relate s+t+u and G}
t+u+v < \mathcal{C}_{61} \frac{\log G}{\log \log G}.
\end{equation}

By the arithmetic-geometric mean inequality we obtain that
\begin{align}\label{Big relation with product of logs}
&\prod_{i=1}^{t}  \log \brac{\text{Nm}_{\QQ}^{K}\brac{\id{p}_{i}}} \cdot \prod_{j=1}^{u} \log \brac{\text{Nm}_{\QQ}^{K}\brac{\id{q}_{j}}} \cdot\prod_{k=1}^{v} \log \brac{\text{Nm}_{\QQ}^{K}\brac{\id{r}_{k}}} \nonumber \\
\leq &\brac{\frac{1}{t+u+v}\brac{\sum_{i=1}^{t}\log \brac{\text{Nm}_{\QQ}^{K}\brac{\id{p}_{i}}} +\sum_{j=1}^{u}\log \brac{\text{Nm}_{\QQ}^{K}\brac{\id{q}_{j}}} +\sum_{k=1}^{v}\log \brac{\text{Nm}_{\QQ}^{K}\brac{\id{r}_{k}}}}}^{t+u+v} \nonumber \\
\leq &\brac{\frac{1}{t+u+v}\sum_{\substack{\id{P} \subset \Of{L}\\ \id{P} \text{prime} \\ \id{P} \mid \brac{abc}\Of{L}}} \log \brac{\text{Nm}_{\QQ}^{L}\brac{\id{P}}}}^{t+u+v} \nonumber \\
\leq &\brac{\frac{\log G}{t+u+v}}^{t+u+v}.
\end{align}

It follows from \eqref{relate s+t+u and G} and \eqref{Big relation with product of logs} that
\begin{equation}\label{Final relation between G and product of norms}
\prod_{i=1}^{t}  \log \brac{\text{Nm}_{\QQ}^{K}\brac{\id{p}_{i}}} \cdot \prod_{j=1}^{u} \log \brac{\text{Nm}_{\QQ}^{K}\brac{\id{q}_{j}}} \cdot\prod_{k=1}^{v} \log \brac{\text{Nm}_{\QQ}^{K}\brac{\id{r}_{k}}} < G^{\mathcal{C}_{62}\frac{\log \log \log G}{\log \log G}}.
\end{equation}
The same logic can be used to show that 
\begin{equation}\label{G and log primes from sunit}
\prod_{\substack{\id{p} \in \Of{L} \\ \id{p} \mid c\Of{L} \\ \id{p} \neq \id{p}_{c}}}\log \text{Nm}_{\QQ}^{L}\brac{\id{p}} < G^{\mathcal{C}_{63}\frac{\log \log \log G}{\log \log G}}, 
\end{equation}
and that
\begin{equation}
    \prod_{\substack{\id{p} \in \Of{L} \\ \id{p} \mid bc\Of{L} }}\log \text{Nm}_{\QQ}^{L}\brac{\id{p}} < G^{\mathcal{C}_{64}\frac{\log \log \log G}{\log \log G}}. 
\end{equation}

Applying these to \eqref{multiplied together (r+s+t) dealt with}, we obtain that
\begin{align}
\brac{\frac{\log H_{L}(a,\,b,\,c)}{\log \log H_{L}(a,\,b,\,c)}}^{3} < &\mathcal{C}_{58}^{t+u+v}  \brac{\log G}^{3} \nonumber \\
&\brac{\text{Nm}_{\QQ}^{L}\brac{\id{p}_{a}} \text{Nm}_{\QQ}^{L}\brac{\id{p}_{b}} \text{Nm}_{\QQ}^{L}\brac{\id{p}_{c}}^{2}\max \curly{\text{Nm}_{\QQ}^{L}\brac{\id{p}_{b}},\,\text{Nm}_{\QQ}^{L}\brac{\id{p}_{c}} }} G^{\mathcal{C}_{65}\frac{\log \log \log G}{\log \log G}}.
\end{align}

Further, from Lemma \ref{primeidealthm} we obtain that
\begin{equation}\label{Constant and G}
\mathcal{C}_{58}^{t+u+v} < G^{\frac{\mathcal{C}_{66}}{\log \log G}}.
\end{equation}
Further, note that $\log G =G^{\frac{\log \log G}{\log G}}$. Thus we obtain that
\begin{align}\label{all but pbpc sorted}
\brac{\frac{\log H_{L}(a,\,b,\,c)}{\log \log H_{L}(a,\,b,\,c)}}^{3} <&   \brac{\text{Nm}_{\QQ}^{L}\brac{\id{p}_{a}} \text{Nm}_{\QQ}^{L}\brac{\id{p}_{b}} \text{Nm}_{\QQ}^{L}\brac{\id{p}_{c}}^{2}\max \curly{\text{Nm}_{\QQ}^{L}\brac{\id{p}_{b}},\,\text{Nm}_{\QQ}^{L}\brac{\id{p}_{c}} }} \nonumber \\
&G^{\mathcal{C}_{67}\brac{\frac{\log \log \log G}{\log \log G}+\frac{1}{\log \log G}+\frac{\log \log G}{\log G}}}.
\end{align}

We take the cube root of both sides, before applying Lemma \ref{last lemma}, obtaining

\begin{align}\label{new theorem equation}
\log H_{L}(a,\,b,\,c) <&  \brac{\text{Nm}_{\QQ}^{L}\brac{\id{p}_{a }}\text{Nm}_{\QQ}^{L}\brac{\id{p}_{b}} \text{Nm}_{\QQ}^{L}\brac{\id{p}_{c}}^{2}\max \curly{\text{Nm}_{\QQ}^{L}\brac{\id{p}_{b}},\,\text{Nm}_{\QQ}^{L}\brac{\id{p}_{c}} }}^{\frac{1}{3}} \nonumber \\ &G^{\mathcal{C}_{68}\brac{\frac{\log \log \log G}{\log \log G}+\frac{1}{\log \log G}+\frac{\log \log G}{\log G}}} .
\end{align} 

Note, the dominant term in the power of $G$ is $\frac{\log \log \log G}{\log \log G}$. Combining this with the above proves Theorem~\ref{main theorem}.

\subsection{Corollaries of Theorem \ref{main theorem}}

In this section we show various corollaries of Theorem 1. The first two corollaries depend on the Class Group of $K$ and the ideals that $\id{p}_{b}$ and $\id{p}_{c}$ lie above.

\begin{corollary}\label{CF1}
Assume that $\text{Nm}_{\QQ}^{L}\brac{\id{p}_{b}} > \text{Nm}_{\QQ}^{L}\brac{\id{p}_{c}}$  and that $\id{p}_{b}$ and $\id{p_{c}}$ both lie over prime ideals of $\Of{K}$ that do not generate the class group of $K$. Then
\begin{align}
\log H_{L}(a,\,b,\,c) &<  G^{\frac{1}{3} + \mathcal{C}_{69}\frac{\log \log \log G}{\log \log G}} .
\end{align} 

\end{corollary}

\begin{proof}
By assumption,  
\begin{align*}
   & \brac{\text{Nm}_{\QQ}^{L}\brac{\id{p}_{a }}\text{Nm}_{\QQ}^{L}\brac{\id{p}_{b}} \text{Nm}_{\QQ}^{L}\brac{\id{p}_{c}}^{2}\max \curly{\text{Nm}_{\QQ}^{L}\brac{\id{p}_{b}},\,\text{Nm}_{\QQ}^{L}\brac{\id{p}_{c}} }} \\
   &= \brac{\text{Nm}_{\QQ}^{L}\brac{\id{p}_{a }}\text{Nm}_{\QQ}^{L}\brac{\id{p}_{b}}^{2} \text{Nm}_{\QQ}^{L}\brac{\id{p}_{c}}^{2}}.
\end{align*}

Recall that in the Hilbert Class Field $L$, a prime ideal $\id{p}$ of $\Of{K}$ splits into $\frac{h_{K}}{P}$ prime ideals of $\Of{L}$, where $P$ is the order of $\squar{\id{p}}$ in the Class Group of $K$. By assumption, there must be at least two prime ideals dividing $b\Of{L}$ with the same norm $\text{Nm}_{\QQ}^{L}\brac{\id{p}_{b}}$, and similarly for $\text{Nm}_{\QQ}^{L}\brac{\id{p}_{c}}$.

It follows that 
\[\brac{\text{Nm}_{\QQ}^{L}\brac{\id{p}_{a }}\text{Nm}_{\QQ}^{L}\brac{\id{p}_{b}}^{2} \text{Nm}_{\QQ}^{L}\brac{\id{p}_{c}}^{2}} \leq G,\]
and the claim follows.
\end{proof}

\begin{corollary}\label{CF2}
Assume that $\text{Nm}_{\QQ}^{L}\brac{\id{p}_{b}} < \text{Nm}_{\QQ}^{L}\brac{\id{p}_{c}} $ and that $\id{p}_{c}$ lies above a prime ideal of $\Of{K}$ that has order greater than 2 in the class group of $K$. Then 
\begin{align}
\log H_{L}(a,\,b,\,c) &<  G^{\frac{1}{3} + \mathcal{C}_{70}\frac{\log \log \log G}{\log \log G}} .
\end{align} 
\end{corollary}
\begin{proof}
By assumption, 
\begin{align*}
 &\brac{\text{Nm}_{\QQ}^{L}\brac{\id{p}_{a }}\text{Nm}_{\QQ}^{L}\brac{\id{p}_{b}} \text{Nm}_{\QQ}^{L}\brac{\id{p}_{c}}^{2}\max \curly{\text{Nm}_{\QQ}^{L}\brac{\id{p}_{b}},\,\text{Nm}_{\QQ}^{L}\brac{\id{p}_{c}} }} \\
 &= \brac{\text{Nm}_{\QQ}^{L}\brac{\id{p}_{a }}\text{Nm}_{\QQ}^{L}\brac{\id{p}_{b}} \text{Nm}_{\QQ}^{L}\brac{\id{p}_{c}}^{3}}.   
\end{align*}

By the comments in the proof of the previous corollary, our assumption here gives us that there are at least 3 prime ideals of $\Of{L}$ dividing $c\Of{L}$ with the same norm, $\text{Nm}_{\QQ}^{L}\brac{\id{p}_{c}}$. It follows that
\[\brac{\text{Nm}_{\QQ}^{L}\brac{\id{p}_{a }}\text{Nm}_{\QQ}^{L}\brac{\id{p}_{b}} \text{Nm}_{\QQ}^{L}\brac{\id{p}_{c}}^{3}} \leq G,\]
and the claim follows.
\end{proof}

The following corollary holds regardless of the class field of $K$.

\begin{corollary}
Assume that $\textrm{Nm}_{\QQ}^{L}\brac{\id{p}_{b}} >\text{Nm}_{\QQ}^{L}\brac{\id{p}_{c}}$. Then
\[\log H_{L}(a,\,b,\,c) <  G^{\frac{2}{3} + \mathcal{C}_{71}\frac{\log \log \log G}{\log \log G}} .\]
\end{corollary}

\begin{proof}
 Note that $\text{Nm}_{\QQ}^{L}\brac{\id{p}_{a }}\text{Nm}_{\QQ}^{L}\brac{\id{p}_{b}} \text{Nm}_{\QQ}^{L}\brac{\id{p}_{c}} \leq G$. Further, by assumption
 \begin{align*}
     \text{Nm}_{\QQ}^{L}\brac{\id{p}_{c}}\max \curly{\text{Nm}_{\QQ}^{L}\brac{\id{p}_{b}},\,\text{Nm}_{\QQ}^{L}\brac{\id{p}_{c}} }&=\text{Nm}_{\QQ}^{L}\brac{\id{p}_{b}} \text{Nm}_{\QQ}^{L}\brac{\id{p}_{c}} \\
     &\leq G.
 \end{align*}
 Thus the corollary follows.
\end{proof}

\begin{corollary}
Assume that $\text{Nm}_{\QQ}^{L}\brac{\id{p}_{a}} > \text{Nm}_{\QQ}^{L}\brac{\id{p}_{b}} > \text{Nm}_{\QQ}^{L}\brac{\id{p}_{c}}.$ Then
\[\log H_{L}(a,\,b,\,c) <  G^{\frac{5}{9} + \mathcal{C}_{72}\frac{\log \log \log G}{\log \log G}}.\]

If $\max\curly{\text{Nm}_{\QQ}^{L}\brac{\id{p}_{b}},\,\text{Nm}_{\QQ}^{L}\brac{\id{p}_{c}}}=\text{Nm}_{\QQ}^{L}\brac{\id{p}_{c}}$ then we obtain that
\[\log H_{L}(a,\,b,\,c) <  G^{\frac{2}{3} + \mathcal{C}_{72}\frac{\log \log \log G}{\log \log G}}.\]
\end{corollary}

\begin{proof}
By assumption, $\text{Nm}_{\QQ}^{L}\brac{\id{p}_{b}}\text{Nm}_{\QQ}^{L}\brac{\id{p}_{c}}\leq G^{\frac{2}{3}}$ and $ \text{Nm}_{\QQ}^{L}\brac{\id{p}_{b}}\leq G^{\frac{1}{2}},\,\text{Nm}_{\QQ}^{L}\brac{\id{p}_{c}} \leq G^{\frac{1}{3}} $. Applying this to Theorem 1 gives both parts of the corollary.
\end{proof}

\begin{remark*}
If we assume that none of $a,\,b,\,c$ are units of $\Of{K}$ then the only assumption we need to obtain the first inequality in the corollary above is that $\max \curly{\text{Nm}_{\QQ}^{L}\brac{\id{p}_{a}},\, \text{Nm}_{\QQ}^{L}\brac{\id{p}_{b}},\, \text{Nm}_{\QQ}^{L}\brac{\id{p}_{c}}} = \text{Nm}_{\QQ}^{L}\brac{\id{p}_{a}}$. This follows as then by assumption, $ \text{Nm}_{\QQ}^{L}\brac{\id{p}_{b}}\leq G^{\frac{1}{3}},\,\text{Nm}_{\QQ}^{L}\brac{\id{p}_{c}} \leq G^{\frac{1}{3}}$. The argument follows. 
\end{remark*}

We now present some corollaries that depend on the value of $\max \curly{\text{Nm}_{\QQ}^{L}\brac{\id{p}_{b}},\,\text{Nm}_{\QQ}^{L}\brac{\id{p}_{c}}}$.

\begin{corollary}
Assume that $\max \curly{\text{Nm}_{\QQ}^{L}\brac{\id{p}_{b}},\,\text{Nm}_{\QQ}^{L}\brac{\id{p}_{c}}} < G^{\alpha}$ with $0<\alpha \leq 1$. Then
\[\log H_{L}(a,\,b,\,c) <  G^{\frac{1+2\alpha}{3} + \mathcal{C}_{73}\frac{\log \log \log G}{\log \log G}}. \]
\end{corollary}

\begin{proof}
Again, $$\text{Nm}_{\QQ}^{L}\brac{\id{p}_{a }}\text{Nm}_{\QQ}^{L}\brac{\id{p}_{b}} \text{Nm}_{\QQ}^{L}\brac{\id{p}_{c}} \leq G.$$ By assumption, 
\[ \text{Nm}_{\QQ}^{L}\brac{\id{p}_{c}}\max \curly{\text{Nm}_{\QQ}^{L}\brac{\id{p}_{b}},\,\text{Nm}_{\QQ}^{L}\brac{\id{p}_{c}} } < G^{2 \alpha}. \]

Applying this to Theorem 1 gives the result.
\end{proof}

\begin{corollary}
Assume that $ \max \curly{\text{Nm}_{\QQ}^{L}\brac{\id{p}_{b}},\,\text{Nm}_{\QQ}^{L}\brac{\id{p}_{c}} } < \brac{\log H_{L} \brac{a,\,b,\,c}}^{\alpha}$ for $0<\alpha < \frac{2}{3}$. Then
\[\log H_{L}(a,\,b,\,c) <  G^{\frac{1}{3-2\alpha} + \mathcal{C}_{74}\frac{\log \log \log G}{\log \log G}}.\]
\end{corollary}

\begin{proof}
Consider \eqref{all but pbpc sorted}. Applying the assumption, we can rewrite this as
\[\brac{\frac{\log H_{L}(a,\,b,\,c)}{\log \log H_{L}(a,\,b,\,c)}}^{3} <   G \brac{\log H_{L}(a,\,b,\,c) }^{2\alpha} G^{\mathcal{C}_{75}\frac{\log \log \log G}{\log \log G}}.\]

Dividing through by $\brac{\log H_{L}(a,\,b,\,c) }^{2\alpha}$ we obtain that
\[\frac{\brac{\log H_{L}(a,\,b,\,c)}^{3-2\alpha} }{\brac{\log \log H_{L}(a,\,b,\,c)}^{3}}< G^{1+\mathcal{C}_{75}\frac{\log \log \log G}{\log \log G}}.\]

Taking the $3-2\alpha$'th root and applying a variant of Lemma \ref{last lemma} gives the result.
\end{proof}

\begin{corollary}
Assume that $\textrm{Nm}_{\QQ}^{L}\brac{\id{p}_{max}}< \brac{\log H_{L}\brac{a,\,b,\,c}}^{\alpha}$ for $0<\alpha < \frac{3}{5}$. Then
\begin{align}
    \log H_{L}\brac{a,\,b,\,c} &< G^{\frac{\mathcal{C}_{76}}{3-5\alpha}\frac{\log \log \log G}{\log \log G} } \nonumber \\
    &= G^{\mathcal{C}_{77}\frac{\log \log \log G}{\log \log G} }.
\end{align}
\end{corollary}
\begin{remark*}
We note that we can write this in the following terms. For any given $\varepsilon > 0 $, given the assumptions of the theorem and corollary there is a computable number $\mathcal{C}_{78}$ such that
\[\log H_{L}\brac{a,\,b,\,c} < G^{\mathcal{C}_{78}\cdot \varepsilon}.\]
\end{remark*}
\begin{proof}
Consider Theorem 1. By assumption, 
\[\log H_{L}\brac{a,\,b,\,c} < \brac{\log H_{L}\brac{a,\,b,\,c}}^{\frac{5 \alpha}{3}} G^{\frac{\log \log \log G}{\log \log G}}.\]
\end{proof}

Dividing through, we obtain that
\[\brac{\log H_{L}\brac{a,\,b,\,c}}^{1-\frac{5 \alpha}{3}} < G^{\frac{\log \log \log G}{\log \log G}}.\]

Take the $\frac{1}{1-\frac{5 \alpha}{3}}$'th root and the result follows.

\section{Method Only Using two $S$-unit bounds}

Part of the difficulty in analysing cases in the previous section comes from the number of prime ideals on the right hand side of \eqref{new theorem equation}. If we only use two $S$-unit bounds then, while in general the bound is worse, it is easier to analyse for corollaries. We now prove Theorem 2, as stated in the introduction.

We follow the main text until \eqref{E}. We then do not use the $S$-unit bound obtained by letting $S$ be equal to the infinite places and finite places corresponding to the prime ideals of $\Of{L}$ dividing $bc \Of{L}$. Following the argument of the main text, we obtain \eqref{H< first bound} and \eqref{H< second bound}. Multiplying these together we obtain that
\begin{align}\label{multiplied together}
&\brac{\frac{\log H_{L}(a,\,b,\,c)}{\log \log H_{L}(a,\,b,\,c)}}^{2} < \mathcal{C}_{79}^{t+u+v} \brac{r+t+u+v}^{7} \brac{\log G}^{2} \brac{\text{Nm}_{\QQ}^{L}\brac{\id{p}_{b}} \text{Nm}_{\QQ}^{L}\brac{\id{p}_{c}}^{2}} \nonumber \\
&\brac{ \prod_{i=1}^{t}  \log \brac{\text{Nm}_{\QQ}^{K}\brac{\id{p}_{i}}} \cdot \prod_{j=1}^{u} \log \brac{\text{Nm}_{\QQ}^{K}\brac{\id{q}_{j}}} \cdot\prod_{k=1}^{v} \log \brac{\text{Nm}_{\QQ}^{K}\brac{\id{r}_{k}}}}^{2} \cdot \prod_{\substack{\id{p} \in \Of{L} \\ \id{p} \mid c\Of{L} \\ \id{p} \neq \id{p}_{c}}}\log \text{Nm}_{\QQ}^{L}\brac{\id{p}}.
\end{align}

From here we follow the arguments of the proof of the main theorem, obtaining
\begin{equation}\label{2 sunits second to last}
\brac{\frac{\log H_{L}(a,\,b,\,c)}{\log \log H_{L}(a,\,b,\,c)}}^{2} <   \brac{\text{Nm}_{\QQ}^{L}\brac{\id{p}_{b}} \text{Nm}_{\QQ}^{L}\brac{\id{p}_{c}}^{2}} G^{\mathcal{C}_{80}\brac{\frac{\log \log \log G}{\log \log G}+\frac{1}{\log \log G}+\frac{\log \log G}{\log G}}}.
\end{equation}

We take the square root of both sides, before applying Lemma \ref{last lemma}, obtaining
\begin{align}\label{theorem 1 equation}
\log H_{L}(a,\,b,\,c) &<  \brac{\text{Nm}_{\QQ}^{L}\brac{\id{p}_{b}} \text{Nm}_{\QQ}^{L}\brac{\id{p}_{c}}^{2}}^{\frac{1}{2}} G^{\mathcal{C}_{81}\brac{\frac{\log \log \log G}{\log \log G}+\frac{1}{\log \log G}+\frac{\log \log G}{\log G}}} \nonumber \\
 &= \text{Nm}_{\QQ}^{L}\brac{\id{p}_{b}}^{\frac{1}{2}} \text{Nm}_{\QQ}^{L}\brac{\id{p}_{c}} G^{\mathcal{C}_{81}\brac{\frac{\log \log \log G}{\log \log G}+\frac{1}{\log \log G}+\frac{\log \log G}{\log G}}}.
\end{align} 
Again, $\frac{\log \log \log G}{\log \log G}$ is the dominant term in the exponent of $G$.

From this point, there are many corollaries we can find, similarly to in the previous section. However, given that there are fewer prime ideal on the right hand side of \eqref{theorem 1 equation}, they are generally easier to prove. Further, Theorem 1 and Theorem 2 are independent, so if $\textrm{Nm}_{\QQ}^{L}\brac{\id{p}_{a}}$ is sufficiently large in comparison to $\textrm{Nm}_{\QQ}^{L}\brac{\id{p}_{b}}$ and $\textrm{Nm}_{\QQ}^{L}\brac{\id{p}_{c}}$, Theorem 2 could give a better bound.

\subsection{Corollaries Of Theorem \ref{thm 1}}

This corollary relies on the class group of $\Of{K}$.
\begin{corollary}\label{Corol 1}
Assume that the prime ideal $\id{r} \subset \Of{K}$ that $\id{p}_{c} \subset \Of{L}$ lies above does not generate the class group of $K$. Then there exits an effectively computable constant $\mathcal{C}_{82}$ such that

\begin{equation}
\log H_{L}(a,\,b,\,c) <   G^{\frac{1}{2} +\mathcal{C}_{82}\frac{\log \log \log G}{\log \log G}}.
\end{equation}

\end{corollary}
\begin{proof}
Let $\id{r}$ be a prime ideal of $\Of{K}$ dividing $c\Of{K}$ such that $\id{p_{c}} \subset \Of{L}$ lies above $\id{r}$. Assume that $\id{r}$ does not generate the class group of $K$. Then in $L=HCF(K)$, $\id{r}$ splits into $h_{K}/P$ prime ideals, where $h_{K}$ is the class number of $K$ and $P$ is the order of $\squar{\id{r}}$ in $C_{K}$ \cite{childress2008class} \cite{neukirch2013algebraic}. As $\id{r}$ does not generate the class group of $K$, then the order of $\squar{\id{r}}$ is at least 2. As $\id{r}$ splits in $\Of{L}$, we know that all prime ideals lying above $\id{r}$ in $\Of{L}$ have the same norm. By assumption, we have at least two such ideals in $\Of{L}$, so $\text{Nm}_{\QQ}^{L}\brac{\id{p}_{b}}\brac{\text{Nm}_{\QQ}^{L}\brac{\id{p}_{c}}}^{2}<G$. More explicitly, there is another prime ideal $\id{p}_{c}'$ of $\Of{L}$ lying above $\id{r}$ such that $\text{Nm}_{\QQ}^{L}\brac{\id{p}_{c}}=\text{Nm}_{\QQ}^{L}\brac{\id{p}_{c}'} $. It then follows from Theorem \ref{thm 1} that
\begin{equation}
\log H_{L}(a,\,b,\,c) <   G^{\frac{1}{2} +\mathcal{C}_{82}\frac{\log \log \log G}{\log \log G}}.
\end{equation}
\end{proof}

The following corollaries give different bounds depending on $\id{p}_{max}$ or $\max \ordid{c}$.
\begin{corollary}\label{Corol 2}
Assume that $\text{Nm}_{\QQ}^{L}\brac{\id{p}_{c}}<G^{\alpha}$ with $0<\alpha <1$, or that $\max \ordid{c}<G^{\alpha}$ with $0<\alpha <1$. Then there exits an effectively computable constant $\mathcal{C}_{83}$ such that
\begin{equation}\label{A}
\log H_{L}(a,\,b,\,c) <   G^{\frac{1+\alpha}{2}+\mathcal{C}_{83}\frac{\log \log \log G}{\log \log G}}.
\end{equation}
Further, if $\max\curly{\text{Nm}_{\QQ}^{L}\brac{\id{p}_{b}},\,\text{Nm}_{\QQ}^{L}\brac{\id{p}_{c}}}<G^{\alpha}$ then
\begin{equation}\label{B}
\log H_{L}(a,\,b,\,c) <  G^{\frac{3 \alpha}{2} +\mathcal{C}_{84}\frac{\log \log \log G}{\log \log G}}.
\end{equation}
\end{corollary}
\begin{remark*}
We note that $\frac{3 \alpha}{2} < 1$ for $\alpha < \frac{2}{3}$, and further that $\frac{3\alpha}{2}< \frac{1+\alpha}{2}$ for $\alpha < \frac{1}{2}$. Thus, our second bound is better than our first given in this corollary for $\alpha < \frac{1}{2}$.
\end{remark*}

\begin{proof}

 We first assume that $\text{Nm}_{\QQ}^{L}\brac{\id{p}_{c}} < G^{\alpha}$ where $\alpha \in (0,\,1)$. Thus 
 \[\text{Nm}_{\QQ}^{L}\brac{\id{p}_{\text{b}}}\text{Nm}_{\QQ}^{L}\brac{\id{p}_{\text{c}}}^{2} <G^{1+\alpha}<G^{2}.\] 
 Thus from Theorem \ref{thm 1}, we obtain that

\begin{equation}\label{Assumption 1 done}
\log H_{L}(a,\,b,\,c) <   G^{\frac{1+\alpha}{2}+\mathcal{C}_{83}\frac{\log \log \log G}{\log \log G}}.
\end{equation}

We now assume that  $\max \ordid{c}<G^{\alpha}$ for some $\alpha \in (0,\,1)$. Then, in place of \eqref{ordid(c) leq bound}, we have that for all $\id{p} \mid c\Of{L}$, $\ordid{c}< G^{\alpha}$. It follows from \eqref{H < ord c} that
\begin{equation}
\log H_{L}(a,\,b,\,c)< \mathcal{C}_{84}\brac{\log G} G^{\alpha}.
\end{equation}
We note that for $\alpha < \frac{1}{2}$, this bound is actually better than the bound that follows.

 As in the proof for the main theorem, \eqref{H< second bound} still holds. Multiplying the above and \eqref{H< second bound} we obtain that
\begin{align}\label{assumption 2.1 combined}
\frac{\brac{\log H_{L}(a,\,b,\,c)}^{2}}{\log \log H_{L}(a,\,b,\,c)}<&\mathcal{C}_{85}^{t+v}\brac{r+t+v+1}^{\frac{7}{2}}\brac{\log G}^{2} G^{\alpha} \text{Nm}_{\QQ}^{L}\brac{\id{p}_{b}}\cdot\text{Nm}_{\QQ}^{L}\brac{\id{p}_{c}} \nonumber \\ 
&\prod_{i=1}^{t} \log \brac{\text{Nm}_{\QQ}^{K}\brac{\id{p}_{i}}}\cdot \prod_{j=1}^{v} \log \brac{\text{Nm}_{\QQ}^{K}\brac{\id{r}_{j}}}\cdot \prod_{\substack{\id{p} \in \Of{L} \\ \id{p} \mid c\Of{L} \\ \id{p} \neq \id{p}_{c}}}\log \text{Nm}_{\QQ}^{L}\brac{\id{p}}.
\end{align}
 
We note that $\text{Nm}_{\QQ}^{L}\brac{\id{p}_{b}}\cdot\text{Nm}_{\QQ}^{L}\brac{\id{p}_{c}}\leq G$. Further, we can use the techniques from above to tidy terms in the same way as we did for the main theorem to obtain

\begin{equation}\label{assumption 2.1 almost done}
\frac{\brac{\log H_{L}(a,\,b,\,c)}^{2}}{\log \log H_{L}(a,\,b,\,c)} <G^{1+\alpha} G^{\mathcal{C}_{86}\frac{\log \log \log G}{\log \log G}}.
\end{equation}
Taking the square root and applying a variant of Lemma \ref{last lemma}, we obtain that

\begin{equation}\label{assumption 2.1 done}
\log H_{L}(a,\,b,\,c) < G^{\frac{1+\alpha}{2}+\mathcal{C}_{83}\frac{\log \log \log G}{\log \log G}}.
\end{equation}

This proves the first part of Corollary \ref{Corol 2}. The further comments follow directly from Theorem \ref{thm 1} when we bound $\text{Nm}_{\QQ}^{L} \brac{\id{p}_{b}}$ and $\text{Nm}_{\QQ}^{L} \brac{\id{p}_{c}}$ above by $G^{\alpha}$. This gives us that 
\begin{equation}
\log H_{L}(a,\,b,\,c) <  G^{\frac{3 \alpha}{2} +\mathcal{C}_{8}4\frac{\log \log \log G}{\log \log G}},
\end{equation}
as claimed, where $\frac{3 \alpha}{2} < 1$ for $\alpha < \frac{2}{3}$, and is a better bound than given above for $\alpha < \frac{1}{2}$.

It also follows directly from Theorem \ref{thm 1} that if $\max \curly{\text{Nm}_{\QQ}^{L}\brac{\id{p}_{\text{b}}},\, \text{Nm}_{\QQ}^{L}\brac{\id{p}_{\text{c}}}} < G^{\alpha}$, then
\[\log H_{L}(a,\,b,\,c) <  G^{\frac{3 \alpha}{2} +\mathcal{C}_{84}\frac{\log \log \log G}{\log \log G}}.\]
\end{proof}

\begin{corollary}\label{Corol 3}
Assume now that $\text{Nm}_{\QQ}^{L}\brac{\id{p}_{c}}<\brac{\log H_{L}(a,\,b,\,c)}^{\alpha}$ with $0<\alpha <1$, or that  \newline $\max \ordid{c}< \brac{\log H_{L}(a,\,b,\,c)}^{\alpha}$ with $0<\alpha <1$. Then there exits an effectively computable constant $\mathcal{C}_{87}$ such that
\begin{equation}
\log H_{L}(a,\,b,\,c)<G^{\frac{1}{2-\alpha}+\mathcal{C}_{87}\frac{\log \log \log G}{\log \log G}}.
\end{equation}
Furthermore, if $\max \curly{\text{Nm}_{\QQ}^{L}\brac{\id{p}_{b}},\,\text{Nm}_{\QQ}^{L}\brac{\id{p}_{c}}}<\brac{\log H_{L}(a,\,b,\,c)}^{\alpha}$ for $\alpha < \frac{2}{3}$, then directly from Theorem \ref{thm 1} we obtain that
\begin{align} \label{C}
\log H_{L}(a,\,b,\,c) &<  G^{\frac{\mathcal{C}_{88}}{2-3\alpha}\frac{\log \log \log G}{\log \log G}} \nonumber \\
&= G^{\mathcal{C}_{89}\frac{\log \log \log G}{\log \log G}}.
\end{align}
\end{corollary}

This is the best bound we achieve in this text.

\begin{remark*}
We note that the second inequality in this corollary gives a sub-exponential bound, an improvement on the bounds given in \cite{StewartYu}\cite{stewart2001abc}. 

To more easily compare with existing results, we note we can slightly weaken this upper bound.  Inequality \eqref{C} implies that given any $\varepsilon > 0$ there exists some computable $\mathcal{C}_{90}$ such that
\[\log H_{L} \brac{a,\,b,\,c} < G^{\mathcal{C}_{90} \cdot \epsilon},\]
where importantly $\mathcal{C}_{90}$ does not depend on $\epsilon$. 

\end{remark*}
\begin{proof}

We first assume that $\text{Nm}_{\QQ}^{L}\brac{\id{p}_{c}}<\brac{\log H_{L}\brac{a,\,b,\,c}}^{\alpha}$ with $\alpha \in (0,1)$. Then from this assumption and \eqref{2 sunits second to last}, we obtain that
\begin{equation}\label{assumption 2 before square root}
\frac{\brac{\log H_{L}(a,\,b,\,c)}^{2-\alpha}}{\brac{\log \log H_{L}(a,\,b,\,c)}^{2}}<G\cdot G^{\mathcal{C}_{91}\frac{\log \log \log G}{\log \log G}}.
\end{equation}
By assumption, $2-\alpha>1$, and we take this root to obtain that
\begin{equation}\label{assumption 2 after root}
\frac{\log H_{L}(a,\,b,\,c)}{\brac{\log \log H_{L}(a,\,b,\,c)}^{\frac{2}{2-\alpha}}}<G^{\frac{1}{2-\alpha}+\mathcal{C}_{92}\frac{\log \log \log G}{\log \log G}}.
\end{equation}
Note that $1>\frac{1}{2-\alpha}>\frac{1}{2}$.
Applying a variant of Lemma \ref{last lemma}, we obtain that
\begin{equation}\label{assumption 2 done}
\log H_{L}(a,\,b,\,c)<G^{\frac{1}{2-\alpha}+\mathcal{C}_{93}\frac{\log \log \log G}{\log \log G}}.
\end{equation}

Now instead assume that $\max \ordid{c}<\brac{\log H_{L}(a,\,b,\,c}^{\alpha}$ for some $\alpha \in (0,\,1)$. Then as above, from \eqref{H < ord c} we obtain that
\begin{equation}
\log H_{L}(a,\,b,\,c)< \mathcal{C}_{94} \log G \brac{\log H_{L}(a,\,b,\,c)}^{\alpha}.
\end{equation}
It immediately follows that
\begin{equation}
\brac{\log H_{L}(a,\,b,\,c)}^{1-\alpha}< \mathcal{C}_{94} \log G.
\end{equation}

Again, we still have \eqref{H< second bound}, as obtained by following the main argument. We multiply \eqref{H< second bound} by the above to obtain
\begin{align}
\frac{\brac{\log H_{L}(a,\,b,\,c)}^{2-\alpha}}{\log \log H_{L}(a,\,b,\,c)}<&\mathcal{C}_{95}^{t+v}\brac{r+t+v+1}^{\frac{7}{2}}\brac{\log G}^{2} \text{Nm}_{\QQ}^{L}\brac{\id{p}_{b}}\cdot\text{Nm}_{\QQ}^{L}\brac{\id{p}_{c}} \nonumber \\ 
&\prod_{i=1}^{t} \log \brac{\text{Nm}_{\QQ}^{K}\brac{\id{p}_{i}}}\cdot \prod_{j=1}^{v} \log \brac{\text{Nm}_{\QQ}^{K}\brac{\id{r}_{j}}}\cdot \prod_{\substack{\id{p} \in \Of{L} \\ \id{p} \mid c\Of{L} \\ \id{p} \neq \id{p}_{c}}}\log \text{Nm}_{\QQ}^{L}\brac{\id{p}}.
\end{align}

As before, we can use the same method of tidying as in the proof of the main theorem to show that

\begin{equation}
\frac{\brac{\log H_{L}(a,\,b,\,c)}^{2-\alpha}}{\log \log H_{L}(a,\,b,\,c)}<G^{1+\mathcal{C}_{96}\frac{\log \log \log G}{\log \log G}}.
\end{equation}

Taking the $2-\alpha$'th root and applying a variant of Lemma \ref{last lemma}, we obtain that
\begin{equation}
\log H_{L}(a,\,b,\,c)<G^{\frac{1}{2-\alpha}+\mathcal{C}_{97}\frac{\log \log \log G}{\log \log G}}.
\end{equation}

This concludes the proof of the first part of Corollary \ref{Corol 3}. 

To see the strongest case, we appeal directly to Theorem \ref{thm 1}. Assume that 
\[\text{Nm} \brac{\id{p}_{\max}}< \log \brac{H_{L}(a,\,b,\,c)}^{\alpha}\]
with $\alpha < \frac{2}{3}$. Then we can bound $\text{Nm}_{\QQ}^{L} \brac{\id{p}_{b}}$ and $\text{Nm}_{\QQ}^{L} \brac{\id{p}_{c}}$ above by $\log \brac{H_{L}(a,\,b,\,c)}^{\alpha}$, obtaining 
\begin{equation}
\log H_{L}(a,\,b,\,c) <  \brac{\log H_{L}(a,\,b,\,c)}^{\frac{3\alpha}{2}} G^{\mathcal{C}_{98}\brac{\frac{\log \log \log G}{\log \log G}+\frac{1}{\log \log G}+\frac{\log \log G}{\log G}}}.
\end{equation}
It then follows that
\begin{equation}
\brac{\log H_{L}(a,\,b,\,c)}^{1-\frac{3\alpha}{2}}< G^{\mathcal{C}_{98}\brac{\frac{\log \log \log G}{\log \log G}+\frac{1}{\log \log G}+\frac{\log \log G}{\log G}}}.
\end{equation}
Taking the $\frac{1}{1-\frac{3\alpha}{2}}$th root gives us the result, namely that
\begin{equation}
\log H_{L}(a,\,b,\,c) <  G^{\frac{\mathcal{C}_{98}}{2-3\alpha}\brac{\frac{\log \log \log G}{\log \log G}+\frac{1}{\log \log G}+\frac{\log \log G}{\log G}}}.
\end{equation}

Finally, we recall that the dominant term in the exponent is $\frac{\log \log \log G}{\log \log G}$, so we obtain that
\begin{equation}
    \log H_{L}(a,\,b,\,c) <  G^{\frac{\mathcal{C}_{88}}{2-3\alpha}\brac{\frac{\log \log \log G}{\log \log G}}}.
\end{equation}

As commented in the statement of the theorem, this is of the form 
\[\log  H_{L}(a,\,b,\,c) < G^{\mathcal{C}_{90}\cdot \epsilon}.\]
\end{proof}

\begin{remark*}
While the assumptions are hard to compare due to their different natures, we can see that for all $\alpha \in (0,\,1)$, $\frac{1+\alpha}{2}\geq \frac{1}{2-\alpha}$. Thus, generally speaking, the bound of Corollary \ref{Corol 3} is better than that of Corollary \ref{Corol 2}. More concretely, given $(a,\,b,\,c)$ that satisfy the assumptions of both Corollary \ref{Corol 2} and \ref{Corol 3}, Corollary \ref{Corol 3} gives a better bound in terms of the radical $G$ than that of Corollary \ref{Corol 2}.
\end{remark*}

\begin{corollary}\label{corol 4}
 Assume that $\text{Nm}_{\QQ}^{L} \brac{\id{p}_{\max}} >G^{\alpha}$ for $\alpha >\frac{1}{3}$, and that $\id{p}_{a}=\id{p}_{\max}$. Then
 
  \begin{equation}
 \log H_{L}(a,\,b,\,c) <   G^{\frac{3-3\alpha}{2}+\mathcal{C}_{99}\frac{\log \log \log G}{\log \log G}}.
 \end{equation}
 If $\text{Nm}_{\QQ}^{L} \brac{\id{p}_{\max}} \leq G^{\frac{1}{3}}$ it follows directly from Theorem \ref{thm 1} that 
 \begin{equation}
\log H_{L}(a,\,b,\,c) <  G^{\frac{1}{2} +\mathcal{C}_{100}\frac{\log \log \log G}{\log \log G}}.
\end{equation}
\end{corollary}

\begin{remark*}
 Note we have the assumption that $\alpha > \frac{1}{3}$ in order to make sure that $\frac{3-3\alpha}{2}<1$.
 
 Further, the second inequality given is the same case as $\alpha = \frac{1}{3}$ in the last part of Corollary \ref{Corol 2}.
\end{remark*}

\begin{proof}
Assume that $\text{Nm}_{\QQ}^{L} \brac{\id{p}_{\max}} >G^{\alpha}$, and assume that $\id{p}_{a}=\id{p}_{\max}$. Then considering \eqref{theorem 1 equation}, we note that \[\text{Nm}_{\QQ}^{L}\brac{\id{p}_{b}} \text{Nm}_{\QQ}^{L} \brac{\id{p}_{c}}^{2} <\brac{G^{1-\alpha}}^{3},\]
so 
\[\brac{\text{Nm}_{\QQ}^{L}\brac{\id{p}_{b}} \text{Nm}_{\QQ}^{L} \brac{\id{p}_{c}}^{2}}^{\frac{1}{2}} < G^{\frac{3-3\alpha}{2}}.\]
It follows that
 \begin{equation}\label{Corollary 4 equation}
 \log H_{L}(a,\,b,\,c) <   G^{\frac{3-3\alpha}{2}+\mathcal{C}_{99}\frac{\log \log \log G}{\log \log G}}.
 \end{equation}
 
 We note that $\frac{3-3\alpha}{2}<1$ for $\alpha> \frac{1}{3}$.
\end{proof}

\begin{corollary}\label{corol 5}
Assume $\text{ord}_{\id{p}_{c}}c < G^{\alpha}$ for $0<\alpha \leq 1$. Then
\begin{equation}
 \log H_{L}\brac{a,\,b,\,c} < G^{\max \curly{\alpha,\, \frac{3}{4}} + \mathcal{C}_{101}\frac{\log \log \log G}{\log \log G}}.
 \end{equation}

\end{corollary}
\begin{proof}
Assume that $\text{ord}_{\id{p}_{c}} c <G^{\alpha}$. 
 
Note that we can write 
\[\max_{\id{p} \mid \prin{c}_{L}} \ordid{c} = \max\curly{\max_{\substack{\id{p} \mid c\Of{L} \\ \id{p} \neq \id{p}_{c}}} \ordid{c},\, \text{ord}_{\id{p}_{c}}\brac{c}}.\]

By assumption, we attain the bound
\begin{equation}\label{Corol 5 initial bound}
\max_{\id{p} \mid \prin{c}_{L}} \ordid{c} = \max\curly{\max_{\substack{\id{p} \mid c\Of{L} \\ \id{p} \neq \id{p}_{c}}} \ordid{c},\, G^{\alpha}}.
\end{equation}

From the above and \eqref{H < ord c}, it follows directly that

\begin{equation}\label{Corol 5 ready to go}
\log H_{L}\brac{a,\,b,\,c} < \mathcal{C}_{102} \log G \max\curly{\max_{\substack{\id{p} \mid c\Of{L} \\ \id{p} \neq \id{p}_{c}}} \ordid{c},\, G^{\alpha}}.
\end{equation}

We now consider cases depending on $\text{Nm}_{\QQ}^{L} \brac{\id{p}_{c}}$.

First assume that $\text{Nm}_{\QQ}^{L} \brac{\id{p}_{\max}}<G^{\frac{1}{2}}$. Then we can directly use Corollary \ref{Corol 2} to obtain the bound given there, namely
\begin{equation}
\log H_{L}(a,\,b,\,c) <   G^{\frac{3}{4}+\mathcal{C}_{103}\frac{\log \log \log G}{\log \log G}}.
\end{equation}
 
 Thus from the above and \eqref{Corol 5 ready to go}, we obtain that
 \begin{equation}
 \log H_{L}\brac{a,\,b,\,c} < G^{\max\curly{\frac{3}{4},\, \alpha}+\mathcal{C}_{104}\frac{\log \log \log G}{\log \log G}}
 \end{equation}

Assume now instead that $\text{Nm}_{\QQ}^{L} \brac{\id{p}_{\max}}\geq G^{\frac{1}{2}}$. It immediately follows that for all other prime ideals $\id{p}$ contributing to $G$, we have that $\text{Nm}_{\QQ}^{L} \brac{\id{p}} < G^{\frac{1}{2}}$. 

Consider now
$\max_{\substack{\id{p} \mid c\Of{L} \\ \id{p} \neq \id{p}_{c}}} \ordid{c}$. We apply Yu's bound as before on this, taking the above comments into consideration. It follows that
\begin{align}
\max_{\substack{\id{p} \mid c\Of{L} \\ \id{p} \neq \id{p}_{c}}} \ordid{c} <& \mathcal{C}_{105}^{t+u}\brac{r+t+u+1}^{7/2} \log \log H_{L}(a,\,b,\,c) \text{Nm}_{\QQ}^{L}\brac{\id{p}} \nonumber \\
&\prod_{i=1}^{t}\log \brac{\text{Nm}_{\QQ}^{K}\brac{\id{p}_{i}}}\cdot \prod_{j=1}^{u} \log \brac{\text{Nm}_{\QQ}^{K}\brac{\id{q}_{j}}},
\end{align}
where $\text{Nm}_{\QQ}^{L} \brac{\id{p}} <G^{\frac{1}{2}}$. Following the same logic as the main text, we obtain that

\begin{equation}\label{Corol 5 ineq 1}
\frac{\log H_{L}(a,\,b,\,c)}{\log \log H_{L}(a,\,b,\,c)}<\mathcal{C}_{106}^{t+u}\brac{r+t+u+1}^{7/2}\log G \cdot  G^{\frac{1}{2}} \prod_{i=1}^{t} \log \brac{\text{Nm}_{\QQ}^{K}\brac{\id{p}_{i}}}\cdot \prod_{j=1}^{u} \log \brac{\text{Nm}_{\QQ}^{K}\brac{\id{q}_{j}}}.
\end{equation}
 
 Note that \eqref{H< second bound} still holds, and $\text{Nm}_{\QQ}^{L} \brac{\id{p}_{b}}\text{Nm}_{\QQ}^{L} \brac{\id{p}_{c}} \leq G$. Multiplying \eqref{H< second bound} and \eqref{Corol 5 ineq 1}, tidying terms as we do in the text, and considering \eqref{Corol 5 ready to go}, we obtain that 
 
 \begin{align}
 \frac{\log H_{L}\brac{a,\,b,\,c}}{\log \log H_{L}\brac{a,\,b,\,c}}<& \max\curly{G^{\frac{3}{4}+\mathcal{C}_{107}\frac{\log \log \log G}{\log \log G}},\, G^{\alpha +\mathcal{C}_{108}\frac{\log \log \log G}{\log \log G}}}.
 \end{align}
 
 More concisely, after applying Lemma \ref{last lemma}, we obtain that 
 \begin{equation}\label{Corol 5 assum 2 done}
 \log H_{L}\brac{a,\,b,\,c} < G^{\max \curly{\alpha,\, \frac{3}{4}} + \mathcal{C}_{101}\frac{\log \log \log G}{\log \log G}}.
 \end{equation}
 
 Thus, in either case depending on $\text{Nm}_{\QQ}^{L} \brac{\id{p}_{\max}}$, we obtain that
 
 \begin{equation}\label{Corol 5 done}
 \log H_{L}\brac{a,\,b,\,c} < G^{\max \curly{\alpha,\, \frac{3}{4}} + \mathcal{C}_{101}\frac{\log \log \log G}{\log \log G}}.
 \end{equation}
\end{proof}

\begin{corollary}\label{corol 6}
Assume that $\text{ord}_{\id{p}_{c}}c < \brac{\log H_{L}\brac{a,\,b,\,c}}^{\alpha}$ for $0<\alpha <1$. Then
 \begin{equation}
  \log H_{L}\brac{a,\,b,\,c} <\max \curly{ G^{\frac{3}{4} + \mathcal{C}_{109}\brac{\frac{\log \log \log G}{\log \log G}}},\, \mathcal{C}_{110} \brac{\log G}^{\frac{1}{1-\alpha}}}.
 \end{equation}

\end{corollary}
\begin{proof}
 Assume that $\text{ord}_{\id{p}_{c}}(c) < \brac{\log H_{L}(a,\,b,\,c)}^{\alpha}$ for some $0<\alpha<1$. As in Corollary \ref{corol 5} it immediately follows that
 \begin{align}
 \max_{\id{p} \mid \prin{c}_{L}} \ordid{c} &= \max\curly{\max_{\substack{\id{p} \mid c\Of{L} \\ \id{p} \neq \id{p}_{c}}} \ordid{c},\, \text{ord}_{\id{p}_{c}}\brac{c}} \nonumber \\
 &\leq \max \curly{\max_{\substack{\id{p} \mid c\Of{L} \\ \id{p} \neq \id{p}_{c}}} \ordid{c},\, \brac{\log H_{L}\brac{a,\,b,\,c}}^{\alpha}}.
 \end{align}
 
 This along with \eqref{H < ord c} implies that 
 \begin{equation}
 \log H_{L}(a,\,b,\,c) < \max \curly{\max_{\substack{\id{p} \mid c\Of{L} \\ \id{p} \neq \id{p}_{c}}} \ordid{c} \log G,\, \mathcal{C}_{111}\brac{\log H_{L}\brac{a,\,b,\,c}}^{\alpha}\log G}.
 \end{equation}
 
 If 
 \[\max \curly{\max_{\substack{\id{p} \mid c\Of{L} \\ \id{p} \neq \id{p}_{c}}} \ordid{c} \log G,\, \brac{\log H_{L}\brac{a,\,b,\,c}}^{\alpha}\log G} = \mathcal{C}_{112}\brac{\log H_{L}\brac{a,\,b,\,c}}^{\alpha}\log G,\]
 then we can see that
 \[\log H_{L}(a,\,b,\,c) < \mathcal{C}_{113}\brac{\log G}^{\frac{1}{1-\alpha}}.\]

We now consider two cases.

In the first case we assume that $\text{Nm}_{\QQ}^{L} \brac{\id{p}_{\max}}<G^{\frac{1}{2}}$. In this case we can appeal directly to Corollary \ref{Corol 2}, obtaining that
 \begin{equation}
 \log H_{L}(a,\,b,\,c) <   G^{\frac{3}{4}+\mathcal{C}_{114}\brac{\frac{\log \log \log G}{\log \log G}}}.
 \end{equation}

For the second case we assume that $\text{Nm}_{\QQ}^{L} \brac{\id{p}_{\max}}\geq G^{\frac{1}{2}}$ and follow the same argument as in Case 2 in Corollary \ref{corol 5}; see section 4.5.2.
 
 As before, we see that for all prime ideals $\id{p} \neq \id{p}_{\max}$ contributing to $G$, we have that $\text{Nm}_{\QQ}^{L}\brac{\id{p}} < G^{\frac{1}{2}}$. Applying Yu's bound on $\max_{\substack{\id{p} \mid c\Of{L} \\ \id{p} \neq \id{p}_{c}}} \ordid{c}$ again, we find that 
 
 \begin{equation}
\max_{\substack{\id{p} \mid c\Of{L} \\ \id{p} \neq \id{p}_{c}}} \ordid{c} < \mathcal{C}_{115}^{s+t}\brac{r+t+u+1}^{7/2} \log \log H_{L}(a,\,b,\,c)\cdot \text{Nm}_{\QQ}^{L}\brac{\id{p}} \prod_{i=1}^{t} \log \brac{\text{Nm}_{\QQ}^{K}\brac{\id{p}_{i}}}\cdot \prod_{j=1}^{u} \log \brac{\text{Nm}_{\QQ}^{K}\brac{\id{q}_{j}}},
\end{equation}
 
 Again, we know that $\text{Nm}_{\QQ}^{L} \brac{\id{p}} <G^{\frac{1}{2}}$. Following the logic of the main text and Section 4.5.2, it again follows that
 
 \begin{equation}
\frac{\log H_{L}(a,\,b,\,c)}{\log \log H_{L}(a,\,b,\,c)}<\mathcal{C}_{116}^{s+t}\brac{r+t+u+1}^{7/2}\log G \cdot  G^{1-\beta} \prod_{i=1}^{t} \log \brac{\text{Nm}_{\QQ}^{K}\brac{\id{p}_{i}}}\cdot \prod_{j=1}^{u} \log \brac{\text{Nm}_{\QQ}^{K}\brac{\id{q}_{j}}}.
\end{equation}
After tidying as we have previously and applying Lemma \ref{last lemma}, it follows that 
 
 \[\log H_{L}(a,\,b,\,c) < G^{\frac{3}{4}+\mathcal{C}_{117}\frac{\log \log \log G}{\log \log G}}.\]
 
 Combining these results, in both cases we obtain that
 
 \begin{equation}
  \log H_{L}\brac{a,\,b,\,c} <\max \curly{ G^{\frac{3}{4} + \mathcal{C}_{109}\frac{\log \log \log G}{\log \log G}},\, \mathcal{C}_{110} \brac{\log G}^{\frac{1}{1-\alpha}}}.
 \end{equation}
\end{proof}

\section{Application of Le Fourn's Method}

A method of Le Fourn \cite{le2020tubular} allows us to improve the above somewhat. We will prove Theorem~3 from the introduction. The important point here is that in our $S$-unit bounds we can use the following lemma to replace the use of the largest prime with the third largest prime in the set $S$. This reduces our reliance on $\id{p}_{a},\, \id{p}_{b},\, \id{p}_{c}$.

\begin{lemma}\label{lefourn}
Let $K$ be a number field of degree $d$ and let $S \subset M_{K}$ containing all the infinite places and a finite number of finite places. Let $s = \abs{S}$ and let $\alpha,\, \beta \in K^{*}$. Consider the $S$-unit equation 
\[\alpha x + \beta y =1\]
with $x,\,y \in \Of{S}^{*}$.

If $S$ contains at most two finite places then all solutions of the above satisfy
\[\max\curly{\h{x},\, \h{y}} \leq \mathcal{C}_{118}\brac{d,\,s}R_{S} \log^{+}\brac{R_{S}}H,\]

where $H=\max \curly{\h{\alpha},\, \h{\beta},\, 1,\, \frac{\pi}{d}}$, $R_{S}$ is the $S$-regulator and $\mathcal{C}_{118}(d,\,s)$ is given in \cite{le2020tubular}.

For any general set of places $S$, all solutions of the above equation satisfy
\[\max\curly{\h{x},\, \h{y}} \leq \mathcal{C}_{119}\brac{d,\,s}P'_{S} R_{S}\brac{1+ \frac{\log^{+}\brac{R_{S}}}{\log^{+}\brac{P'_{S}}}}H,\]
with $P'_{S}$ the third largest value of the norms of ideals coming from the finite places of $S$ and $\mathcal{C}_{119}\brac{d,\,s}$ given in \cite{le2020tubular}. If there are fewer than 3 finite places in $S$ then we take $P'_{S}=1$.

\end{lemma}
\begin{proof}
This is Theorem 1.4 of \cite{le2020tubular}. We note the constants are taken from \cite{gyHory2006bounds}.
\end{proof}

When we apply Le Fourn's lemma in the places we previously applied Lemma \ref{sunitbound} we attain (after moving things into the constant) essentially the same bounds with $\id{p}_{a},\, \id{p}_{b},\, \id{p}_{c}$ replaced by $\id{p}_{a}',\, \id{p}_{b}'$ and $\id{p}_{c}'$ where $\id{p}_{a}'$ is the prime ideal of third largest norm dividing $a \Of{L}$ and similarly for $\id{p}_{b}'$ and $\id{p}_{c}'$. If fewer than three prime ideals divide $a,\, b$ or $c$ then we define the corresponding norm to be 1.

We follow the proof of the main theorem, but we replace any use of Lemma \ref{sunitbound} with Lemma \ref{lefourn}. For the most part, all that changes is any occurrence of $\id{p}_{a},\, \id{p}_{b}$ and $\id{p}_{c}$ arising from the use of Lemma \ref{sunitbound} is replaced by $\id{p}_{a}',\, \id{p}_{b}'$ and $\id{p}_{c}'$. We follow the line of reasoning from the main text up until \eqref{first sunit bound}. For this first application of $S$-units, where we have no finite places, we continue to use Lemma \eqref{sunitbound} as it is simpler in this case than Lemma \ref{lefourn}. As before, we obtain \eqref{D}. 

As before, now let $S=S_{\infty} \cup \curly{\id{p} \,:\, \id{p}\mid c \Of{L}}.$ Applying Lemma \ref{lefourn} to 
\[-\frac{u_{a}a'}{u_{c}c'} -\frac{u_{b}b'}{u_{c}c'} = 1,\]
we obtain that
\begin{align}
    \max \curly{\h{-\frac{u_{a}}{u_{c}c'}},\, \h{-\frac{u_{b}}{u_{c}}}} &< \mathcal{C}_{120}P'_{S}R_{S}\brac{1+\frac{\log^{+}R_{S}}{\log^{+}P'_{S}}}\max\curly{\h{a'},\, \h{b'},\,1,\,\frac{\pi}{d}} \nonumber \\
    &< \mathcal{C}_{121}\text{Nm}_{\QQ}^{L}\brac{\id{p}'_{c}}R_{S}^{2} \max{\curly{\h{a'},\, \h{b'},\,1,\,\frac{\pi}{d}}} \nonumber \\
    &< \mathcal{C}_{122}\text{Nm}_{\QQ}^{L}\brac{\id{p}'_{c}} \brac{\prod_{\substack{\id{p} \subset \Of{L}\\ \id{p}\mid c\Of{L}}}\log \text{Nm}_{\QQ}^{L}\brac{\id{p}}}^{2}\max \curly{\frac{\pi}{2}\h{a'},\, \frac{\pi}{2}\h{b'},\, \frac{\pi}{2}} \nonumber \\
    &=\mathcal{C}_{123}\text{Nm}_{\QQ}^{L}\brac{\id{p}'_{c}} \brac{\prod_{\substack{\id{p} \subset \Of{L}\\ \id{p}\mid c\Of{L}}}\log \text{Nm}_{\QQ}^{L}\brac{\id{p}}}^{2} \h{b'},
\end{align}
where the line of reasoning about $ \max{\curly{\h{a'},\, \h{b'},\,1,\,\frac{\pi}{d}}}$ follows from assumption \eqref{1<ha<hb<hc}. We are thus able to replace \eqref{E} with
\begin{equation}
    \max\curly{\h{\frac{u_{a}a'}{u_{c}c'}},\, \h{\frac{u_{b}b'}{u_{c}c'}}} < \mathcal{C}_{124}\text{Nm}_{\QQ}^{L}\brac{\id{p}'_{c}} \brac{\prod_{\substack{\id{p} \subset \Of{L}\\ \id{p}\mid c\Of{L}}}\log \text{Nm}_{\QQ}^{L}\brac{\id{p}}}^{2} \h{b'}.
\end{equation}

As before, we now let $S=S_{_\infty}\cup\curly{\id{p}\, :\, \id{p} \mid bc\Of{L}}$. Applying Lemma 8 again, following the same method as above, in place of \eqref{F} we obtain that
\begin{equation}
    \max \curly{\h{-\frac{u_{a}}{u_{c}c'}},\, \h{-\frac{u_{b}b'}{u_{c}c'}}} < \mathcal{C}_{125}\text{Nm}_{\QQ}^{L}\brac{\id{q}} \brac{\prod_{\substack{\id{p} \subset \Of{L}\\ \id{p}\mid bc\Of{L}}}\log \text{Nm}_{\QQ}^{L}\brac{\id{p}}}^{2} \h{a'},
\end{equation}
where $\id{q}$ is the prime ideal of $\Of{L}$ of third largest norm dividing $bc\Of{L}$. We note that this is not necessarily $\id{p}_{b}$ or $\id{p}_{c}$, though it may have the same norm as one of them, and could indeed be either of them.

We now follow the argument of the main text again, using the above inequalities in place of \eqref{E} and \eqref{F} as necessary, and we end up obtaining

\begin{align}\label{new theorem equation''}
\log H_{L}(a,\,b,\,c) &<  \brac{\text{Nm}_{\QQ}^{L}\brac{\id{p}_{a }}\text{Nm}_{\QQ}^{L}\brac{\id{p}_{b}} \text{Nm}_{\QQ}^{L}\brac{\id{p}_{c}} \text{Nm}_{\QQ}^{L}\brac{\id{p}_{c}'} \text{Nm}_{\QQ}^{L}\brac{\id{q}}}^{\frac{1}{3}} G^{\mathcal{C}_{126}\brac{\frac{\log \log \log G}{\log \log G}+\frac{1}{\log \log G}+\frac{\log \log G}{\log G}}}
\end{align} 
in place of \eqref{new theorem equation}. As before, $\frac{\log \log \log G}{\log \log G}$ is the dominant term in the exponent of $G$, so we can write

\begin{align}\label{new theorem equation'''}
\log H_{L}(a,\,b,\,c) &<  \brac{\text{Nm}_{\QQ}^{L}\brac{\id{p}_{a }}\text{Nm}_{\QQ}^{L}\brac{\id{p}_{b}} \text{Nm}_{\QQ}^{L}\brac{\id{p}_{c}} \text{Nm}_{\QQ}^{L}\brac{\id{p}_{c}'} \text{Nm}_{\QQ}^{L}\brac{\id{q}}}^{\frac{1}{3}} G^{\mathcal{C}_{127}\frac{\log \log \log G}{\log \log G}}.
\end{align} 

We explore some cases. If $\id{q}=\id{p}_{b}'$ then 
\[\text{Nm}_{\QQ}^{L}\brac{\id{p}_{a }}\text{Nm}_{\QQ}^{L}\brac{\id{p}_{b}} \text{Nm}_{\QQ}^{L}\brac{\id{p}_{c}} \text{Nm}_{\QQ}^{L}\brac{\id{p}_{c}'} \text{Nm}_{\QQ}^{L}\brac{\id{p}_{b}'}\leq G.\] 

If $\id{q}=\id{p}_{c}'$ then there exists a prime ideal $\id{p}_{c}''$, the prime of second largest norm dividing $c \Of{L}$. Note that $\text{Nm}_{\QQ}^{L}\brac{\id{p}_{c}'}\leq \text{Nm}_{\QQ}^{L}\brac{\id{p}_{c}}'' \leq \text{Nm}_{\QQ}^{L}\brac{\id{p}_{c}}$, and all these primes divide $abc \Of{L}$ so their norms contribute to $G$. Thus, in this case we obtain that

\begin{align*}
    \text{Nm}_{\QQ}^{L}\brac{\id{p}_{a }}\text{Nm}_{\QQ}^{L}\brac{\id{p}_{b}} \text{Nm}_{\QQ}^{L}\brac{\id{p}_{c}} \text{Nm}_{\QQ}^{L}\brac{\id{p}_{c}'} \text{Nm}_{\QQ}^{L}\brac{\id{p}_{c}'}&< \text{Nm}_{\QQ}^{L}\brac{\id{p}_{a }}\text{Nm}_{\QQ}^{L}\brac{\id{p}_{b}} \text{Nm}_{\QQ}^{L}\brac{\id{p}_{c}} \text{Nm}_{\QQ}^{L}\brac{\id{p}_{c}'} \text{Nm}_{\QQ}^{L}\brac{\id{p}_{c}''} \\
    &\leq G.
\end{align*}
 We have dealt with the cases where $\id{q} = \id{p}_{b}'$ and $\id{q} = \id{p}_{c}'$. Using the notation above, there are four further possibilities for $\id{q}$, namely $\id{p}_{b},\, \id{p}_{c},\, \id{p}_{b}'',\, \id{p}_{c}''$. If $\id{q} = \id{p}_{b}''$ or $\id{p}_{c}''$ then substituting into the above expression, it follows from the definition of $G$ that $\text{Nm}_{\QQ}^{L}\brac{\id{p}_{a }}\text{Nm}_{\QQ}^{L}\brac{\id{p}_{b}} \text{Nm}_{\QQ}^{L}\brac{\id{p}_{c}} \text{Nm}_{\QQ}^{L}\brac{\id{p}_{c}'} \text{Nm}_{\QQ}^{L}\brac{\id{q}} < G$.
 
 On the other hand, if $\id{q}= \id{p}_{b}$ or $\id{p}_{c}$, we can still upper bound this by $G$. Assume $\id{q}= \id{p}_{b}$. Then we deduce that $\text{Nm}_{\QQ}^{L}\brac{\id{p}_{b}} \leq \text{Nm}_{\QQ}^{L}\brac{\id{p}_{c}'} \leq \text{Nm}_{\QQ}^{L}\brac{\id{p}_{c}}$. It follows then that
 
 \begin{align*}
         \text{Nm}_{\QQ}^{L}\brac{\id{p}_{a }}\text{Nm}_{\QQ}^{L}\brac{\id{p}_{b}} \text{Nm}_{\QQ}^{L}\brac{\id{p}_{c}} \text{Nm}_{\QQ}^{L}\brac{\id{p}_{c}'} \text{Nm}_{\QQ}^{L}\brac{\id{p}_{b}}&< \text{Nm}_{\QQ}^{L}\brac{\id{p}_{a }}\text{Nm}_{\QQ}^{L}\brac{\id{p}_{b}} \text{Nm}_{\QQ}^{L}\brac{\id{p}_{c}} \text{Nm}_{\QQ}^{L}\brac{\id{p}_{c}'} \text{Nm}_{\QQ}^{L}\brac{\id{p}_{c}'} \\
    &\leq G.
 \end{align*}
The argument is symmetric so applies if $\id{q}= \id{p}_{b}$. Thus in all cases we obtain that
\begin{align}\label{new theorem equation Le Fourn}
\log H_{L}(a,\,b,\,c) &<  G^{\frac{1}{3}+ \mathcal{C}_{128}\frac{\log \log \log G}{\log \log G}}.
\end{align} 

We note that for given $a,\,b,\,c$, once we know the prime ideals dividing $a\Of{K},\,b\Of{K}$ and $c\Of{K}$, the inequality \eqref{new theorem equation'''} may be stronger than that given in \eqref{new theorem equation Le Fourn}.

\section{Some Remarks}

A combination of methods and results by Gy\"{o}ry and Yu \cite{gyHory2006bounds} and Gy\"{o}ry \cite{gyHory2008abc}, \cite{gyHory2019bounds} with the method of Le Fourn \cite{le2020tubular} can be used directly to find results over the base field, as done by Gy\"{o}ry in \cite{gyorynewpaper}. Further, in terms of $S$, Gy\"{o}ry improved the $S$-unit bound given by Le Fourn. Gy\"ory's result regarding the $abc$ conjecture is as follows.

Let $K$ be a number field and let $a,\,b,\, c:=a+b$ belong to $K^{*}$. Define
\[N_{K} = \prod_{\upsilon} \textrm{Nm}_{\QQ}^{K}\brac{\id{p}}^{\ordp{\id{p}}},\]
where $\upsilon$ is taken from the set of finite places such that $\abs{a}_{\upsilon},\, \abs{b}_{\upsilon}$ and $\abs{c}_{\upsilon}$ are not all equal, and $p$ is the rational prime such that $\id{p} \cap \ZZ = p$. Then for all $\varepsilon$ there is a computable constant $\mathcal{C}_{129}$ depending only on $d=\squar{K:\QQ},\, \Delta_{K}$ and $\varepsilon$ such that
\[\log H_{K}\brac{a,\,b,\,c} < \mathcal{C}_{129}G_{K}^{\frac{1}{3}+\epsilon}.\]

Gy\"ory's combination of his method with that of Le Fourn's enables him to state his results entirely over the base field $K$. On the other hand, the dependence on the norms of prime ideals in this paper allows us to state corollaries depending on these norms, leading to the sub-exponential bound for example. Further, we believe they may allow some attack at open problems such as the smooth $abc$ conjecture \cite{lagarias2011smooth}.

For all these results, we have considered them in terms of $\log H_{L}\brac{a,\,b,\,c}$. We note that $H_{L}\brac{a,\,b,\,c} = H_{K}\brac{a,\,b,\,c}^{h_{K}}$, as $h_{K}=\squar{L:K}$. Thus, as $h_{K}$ depends on the field, after taking the logarithm we can incorporate the $h_{K}$ into our computable constant and have the height in terms of the base field $K$. However, so far we have been unable to do the same for the radical $G$.

\section{Application to Effective Skolem-Mahler-Lech Problem}

In this section we will use our main result to allow us to determine whether a linear recurrence sequence of degree three with no repeated roots of the characteristic polynomial has zeroes. As noted in the introduction, there exists an algorithm to determine whether there are periodic zeroes, so we are concerned with the case when there are only potentially finitely many zeroes.

\subsection{Case Where all Terms are Coprime}

Consider a linear recurrence sequence of the following form:
\[a_{n}=c_{1} a_{n-1} +c_{2} a_{n-2} + c_{3} a_{n-3},\]
where the values of $a_{0},\, a_{1}$ and $a_{2}$ are known. We form the characteristic polynomial of the sequence
\[x^{3}-c_{1} x^{2} -c_{2} x -c_{3} \]
and assume that this has distinct roots $r_{1},\,r_{2},\,r_{3}$. Let $K=\QQ\brac{r_{1},\,r_{2},\,r_{3}}$. We further assume the roots are pairwise coprime when considered as principal ideals of the ring of integers $\Of{K}$. 

By our assumptions, we know that we can write
\[a_{n}=k_{1}r_{1}^{n} + k_{2} r_{2}^{n} + k_{3}r_{3}^{n}\]
where $k_{1},\,k_{2},\,k_{3}$ are constants depending on $a_{0},\, a_{1}$ and $a_{2}$. We further assume $k_{1},\, k_{2}$ and $k_{3}$ are coprime, but we will look at ways to try and deal with this when not coprime later.

For ease, we use the result obtained by using Le Fourn's Lemma, that is the inequality given at \eqref{new theorem equation Le Fourn}. Assume there exists an $n$ such that $a_{n}=0$. Explicitly, 
\[0=k_{1}r_{1}^{n} + k_{2} r_{2}^{n} + k_{3}r_{3}^{n}.\]
We are in a position to use the result. Let $L = HCF(K)$ and define $G$ as above. Then, 
\[\log H\brac{k_{1}r_{1}^{n},\,k_{2} r_{2}^{n},\,k_{3}r_{3}^{n}} < G^{\frac{1}{3}+\mathcal{C}_{128}\frac{\log \log \log G}{\log \log G}}.\]
Without loss of generality, assume that
\[\h{r_{1}}\leq \h{r_{2}} \leq \h{r_{3}}.\] 
Note that 
\[H\brac{k_{1}r_{1}^{n},\,k_{2} r_{2}^{n},\,k_{3}r_{3}^{n}}=H\brac{\frac{k_{1}}{k_{3}}r_{1}^{n},\,\frac{k_{2}}{k_{3}} r_{2}^{n},\,r_{3}^{n}}.\]
Further, by comparing definitions, \[\h{r_{3}^{n}} \leq \log H\brac{\frac{k_{1}}{k_{3}}r_{1}^{n},\,\frac{k_{2}}{k_{3}} r_{2}^{n},\,r_{3}^{n}}.\] Moreover, $\h{r_{3}^{n}} ~= n \h{r_{3}}$ \cite{Waldschmidt}. Combining all this we obtain that
\[n\h{r_{3}} < G^{\frac{1}{3}+\mathcal{C}_{128}\frac{\log \log \log G}{\log \log G}}.\]
It follows that 
\[n < \frac{G^{\frac{1}{3}+\mathcal{C}_{128}\frac{\log \log \log G}{\log \log G}}}{\h{r_{3}}},\]
giving an upper bound for $n$. 

Explicitly, given a recurrence relation satisfying the given conditions, we first check whether there are any zeroes in arithmetic progressions. If so, we are done. If not, we apply the above method, which gives an upper bound for the maximal value of $n$ such that $a_{n}=0$. We numerically check the values of $a_{x}$ for $x$ less than the obtained upper bound. This answers the question as to whether the recurrence sequence has a zero. 

\begin{example}
Consider the linear recurrence sequence with $a_{0}=31,\, a_{1}=112,\, a_{2}=452$ and \[a_{n}=10a_{n-1}-31a_{n-2}+30a_{n-3}.\]
This sequence has characteristic polynomial 
\[x^{3}-10x^{2}+31x-30,\]
with roots $2,\,3$ and $5$. Thus, 
\[a_{n}=k_{1}2^{n}+k_{2}3^{n}+k_{3}5^{n},\]
where $k_{1},\,k_{2}$ and $k_{3}$ are to be found. They are found to be $k_{1}=7,\, k_{2}=11$ and $k_{3}=13$, so
\[a_{n}=7\cdot2^{n}+11 \cdot3^{n}+13 \cdot 5^{n}.\]

This means $G=2\cdot3\cdots5\cdot7\cdot11\cdot13=30030$. The largest logarithmic height of the roots is $\h{5}=\log 5$. It follows that if $a_{n}=0$, then
\begin{align*}
    n &< \frac{30030^{\frac{1}{3}+\mathcal{C}_{128}\frac{\log \log \log 30030}{\log \log 30030}}}{\log 5}\\
    &< 20 \cdot 43^{\mathcal{C}_{130}}.
\end{align*}
In principle, $\mathcal{C}_{128}$ and $\mathcal{C}_{130}$ can be computed following the proof given in this paper. This gives an upper bound for $n$.
\end{example}
In this example, once we derive $a_{n}=7\cdot2^{n}+11 \cdot3^{n}+13 \cdot 5^{n}$, it is clear there are no zeroes. With more complicated examples, it may not be so obvious.

\subsection{Case Where Terms are Not Coprime}

Assume we have a linear recurrence relation as above with characteristic polynomial $f(x)$ with roots $r_{1},\, r_{2},\, r_{3}$. Assume there are constants $k_{1},\, k_{2},\, k_{3}$ so
\[a_{n}=k_{1}r_{1}^{n}+k_{2}r_{2}^{n}+k_{3}r_{3}^{n}.\]

We assume nothing about coprimeness. If they are all coprime, we're done as above. We thus assume $k_{1}r_{1}^{n},\, k_{2}r_{2}^{n},\, k_{3}r_{3}^{n}$ are not coprime. 

If there exists an $n$ such that $a_{n}=0$, then the same prime ideal must divide all 3 terms. We can see this as if
\[0=k_{1}r_{1}^{n}+k_{2}r_{2}^{n}+k_{3}r_{3}^{n},\]
then \[-k_{1}r_{1}^{n}=k_{2}r_{2}^{n}+k_{3}r_{3}^{n},\]

and it follows if a prime ideal divides two of these as ideals, it has to divide the third.

More rigorously, 
\begin{equation}\label{valuation property}
\ordid{a+b}\geq \min \curly{\ordid{a},\, \ordid{b}},    
\end{equation}

with equality when $\ordid{a}\neq \ordid{b}$. The claim directly follows from this. It also follows from this that at least two of the terms are divisible by the prime ideal to the same order.

We consider these as ideals of $\Of{K}$. Assume that $\id{q}$ is a prime ideal of $\Of{K}$ dividing all three terms, and that  $\text{ord}_{\id{q}}\brac{k_{1}r_{1}^{n}} = \text{ord}_{\id{q}}\brac{k_{2}r_{2}^{n}} = l$. We write
\begin{align*}
    k_{1}r_{1}^{n}\Of{K} &= \id{p}_{k_{1},\,1}^{e_{k_{1},\,1}}\cdots \id{p}_{k_{1},\,a}^{e_{k_{1},\,a}} \id{p}_{r_{1},\,1}^{n \cdot e_{r_{1},\,1}}\cdots \id{p}_{r_{1},\,b}^{n \cdot e_{r_{1},\,b}} \id{q}^{l}\\
    k_{2}r_{2}^{n} \Of{K}&= \id{p}_{k_{2},\,1}^{e_{k_{2},\,1}}\cdots \id{p}_{k_{2},\,a}^{e_{k_{2},\,c}} \id{p}_{r_{2},\,1}^{n \cdot e_{r_{2},\,1}}\cdots \id{p}_{r_{2},\,b}^{n \cdot e_{r_{2},\,d}} \id{q}^{l}\\
     k_{3}r_{3}^{n} \Of{K} &= \id{p}_{k_{3},\,1}^{e_{k_{3},\,1}}\cdots \id{p}_{k_{3},\,f}^{e_{k_{3},\,f}} \id{p}_{r_{3},\,1}^{n \cdot e_{r_{3},\,1}}\cdots \id{p}_{r_{3},\,b}^{n \cdot e_{r_{3},\,g}} \id{q}^{m},
\end{align*}
where these ideals are prime ideals of $\Of{K}$.
    
Note, it may be the case that $l=m$. Also, if $\id{q}\mid r_{1}^{n}$, this implies that $l=an$ for some $a$, but we will see this doesn't matter for the argument. Finally, it may be that there is a further prime ideal that divides all three terms;p if so, we apply the following process iteratively on all prime ideals dividing all three terms. 

We now move the the Hilbert Class Field $L$. All the ideals above are principal as ideals of $\Of{L}$, so we can write 
\begin{align*}
    k_{1}r_{1}^{n}&= u_{1}{p}_{k_{1},\,1}^{e_{k_{1},\,1}}\cdots {p}_{k_{1},\,a}^{e_{k_{1},\,a}} {p}_{r_{1},\,1}^{n \cdot e_{r_{1},\,1}}\cdots {p}_{r_{1},\,b}^{n \cdot e_{r_{1},\,b}} {q}^{l}\\
    k_{2}r_{2}^{n} &= u_{2}{p}_{k_{2},\,1}^{e_{k_{2},\,1}}\cdots {p}_{k_{2},\,a}^{e_{k_{2},\,c}} {p}_{r_{2},\,1}^{n \cdot e_{r_{2},\,1}}\cdots {p}_{r_{2},\,b}^{n \cdot e_{r_{2},\,d}} {q}^{l}\\
     k_{3}r_{3}^{n}  &= u_{3}{p}_{k_{3},\,1}^{e_{k_{3},\,1}}\cdots {p}_{k_{3},\,f}^{e_{k_{3},\,f}} {p}_{r_{3},\,1}^{n \cdot e_{r_{3},\,1}}\cdots {p}_{r_{3},\,b}^{n \cdot e_{r_{3},\,g}} {q}^{m},
\end{align*}
where the terms on the right hand side are all elements of $L$ that generate the relevant principal ideals. Thus, we can now write

\begin{align*}
a_{n}=&u_{1}{p}_{k_{1},\,1}^{e_{k_{1},\,1}}\cdots {p}_{k_{1},\,a}^{e_{k_{1},\,a}} {p}_{r_{1},\,1}^{n \cdot e_{r_{1},\,1}}\cdots {p}_{r_{1},\,b}^{n \cdot e_{r_{1},\,b}} {q}^{l}\\ &+u_{2}{p}_{k_{2},\,1}^{e_{k_{2},\,1}}\cdots {p}_{k_{2},\,a}^{e_{k_{2},\,c}} {p}_{r_{2},\,1}^{n \cdot e_{r_{2},\,1}}\cdots {p}_{r_{2},\,b}^{n \cdot e_{r_{2},\,d}} {q}^{l}\\&+u_{3}{p}_{k_{3},\,1}^{e_{k_{3},\,1}}\cdots {p}_{k_{3},\,f}^{e_{k_{3},\,f}} {p}_{r_{3},\,1}^{n \cdot e_{r_{3},\,1}}\cdots {p}_{r_{3},\,b}^{n \cdot e_{r_{3},\,g}} {q}^{m}.    
\end{align*}

By \eqref{valuation property}, we see that $m\geq l$. We thus divide through the above equation by $q^{l}$. After repeating this for all prime ideals dividing all three terms, the remaining terms on the right hand side will all be coprime. We now assume that there exists an $n$ such that $a_{n}=0$, and we are in the same position as section 8.1, and the argument follows identically.

\begin{remark*}
We note that this application also follows from Gy\"ory's result \cite{gyorynewpaper}.
\end{remark*}

\section{Smooth Solutions to the $abc$ Conjecture}
In this section we will prove Theorem \ref{XYZ improvement}, as given in the introduction. We will first prove the following lemma.
\begin{lemma}\label{G<es}
Let $\brac{X,\,Y,\,Z} \in \ZZ^{3}$ be a triple with smoothness $S\brac{X,\,Y,\,Z}$ and radical $G\brac{X,\,Y,\,Z}$ defined as above. Then
\begin{equation}\label{bound on G smooth}
    G\brac{X,\,Y,\,Z} \leq e^{3 S\brac{X,\,Y,\,Z}}.
\end{equation}
\end{lemma}

\begin{proof}
From \cite{rosser1941explicit}, we know that for $n\geq 6$, 
\[\frac{p_{n}}{n}<\log n + \log \log n,\]
where $p_{n}$ denotes the n'th rational prime.

It follows that for $n \geq 6$, 
\[p_{n}\leq 2n \log n.\]
Indeed, computationally we can check primes $2,\,3,\,5,\,7,\,11$ and we find that for $n>2$ the above inequality holds.

Further, by Rosser's Theorem \cite{rosser1939n},
\[p_{n} > n \log n.\]

Thus, it follows that the product of the first $k$ primes satisfies the following inequality:
\begin{align}
    \prod_{i=1}^{k}p_{i} & \leq 2\cdot 3 \cdot\prod_{i=3}^{k}2i \log i \nonumber \\
    &= 2\cdot 3 \cdot 2^{k-2}\brac{4\cdot 5 \cdots k} \cdot \prod_{i=3}^{k}\log i \nonumber \\
    &= 2^{k-2}\cdot k! \cdot \prod_{i=3}^{k} \log i \nonumber \\
    & \leq 2^{k} \cdot k^{k} \cdot \prod_{i=3}^{k} \log i.
\end{align}
We now take logs of each side of the inequality attaining
\begin{align}
    \log \brac{\prod_{i=1}^{k}p_{i}} & \leq \log \brac{2^{k} \cdot k^{k} \cdot \prod_{i=3}^{k} \log i} \nonumber \\
    &= k \log 2 +k \log k + \log \brac{\prod_{i=3}^{k} \log i} \nonumber \\
    &= k \log 2 + k \log k +\sum_{i=3}^{k} \log \log i \nonumber \\
    &\leq k \log 2 + k\log k + k \log \log k \nonumber \\
    & \leq k \log 2 + k\log k +k \log k \nonumber \\
    & \leq 3k \log k \nonumber \\
    &\leq 3 p_{k}
\end{align}
where the last line follows from Rosser's Theorem.

It thus follows that for a triple of pairwise coprime integers satisfying $X+Y=Z$ with smoothness $S\brac{X,\,Y,\,Z}$, 
\begin{align}
    G\brac{X,\,Y,\,Z}&\leq \prod_{\substack{p \textrm{ prime}\\ p \leq S\brac{X,\,Y,\,Z}}}p \nonumber \\
    & \leq e^{3 S\brac{X,\,Y,\,Z}},
\end{align}
by the above inequality.
\end{proof}

We are now in a position to prove Theorem~\ref{XYZ improvement}.
\begin{proof}
By Northcott's Theorem, we can assume in the following that $H\brac{X,\,Y,\,Z}>B$ for any given bound $B$ as we will only be excluding finitely many possible solutions $\brac{X,\,Y,\,Z}$ satisfying \eqref{S assumption}.

For ease of notation, write
\begin{equation}\label{Definition of T}
    T:=\log \log H\brac{X,\,Y,\,Z}.
\end{equation}

By assumption, we have that
\begin{equation}\label{xyz ineq}
S\brac{X,\,Y,\,Z} < T \frac{\log T}{\log \log T \phi\brac{T}}.    
\end{equation}

We study triples $\brac{X,\,Y,\,Z}$ satisfying \eqref{xyz ineq}, and note that for a sufficiently large $H\brac{X,\,Y,\,Z}$,
\[T \frac{\log T}{\log \log T \phi\brac{T}} < \brac{\log H\brac{X,\,Y,\,Z}}^{\frac{1}{2}},\]
so we can apply Corollary 10. We note the choice of the exponent to be $\frac{1}{2}$ is incidental; indeed any exponent less than $\frac{2}{3}$ could have been chosen. Further, the Hilbert Class Field of $\QQ$ is itself $\QQ$ \cite{childress2008class}, so the radical $G$ in this case is defined over $\QQ$ and coincides with the radical given in \cite{lagarias2011smooth}.

By Corollary 10, we know that
\[\log H\brac{X,\,Y,Z} < G^{\mathcal{C}_{89}\frac{\log \log \log G}{\log \log G}}.\]

From the upper bound for $G$ given at \eqref{bound on G smooth} in Lemma \ref{G<es}, we obtain from the above that
\begin{align}
\log H\brac{X,\,Y,\,Z} &< e^{3\mathcal{C}_{89} S\brac{X,\,Y,\,Z}\frac{\log \log \log e^{3 S\brac{X,\,Y,\,Z}}}{\log \log e^{3 S\brac{X,\,Y,\,Z}}}} \nonumber \\
&=e^{\mathcal{C}_{131} S\brac{X,\,Y,\,Z}\frac{\log \log 3 S\brac{X,\,Y,\,Z}}{\log 3 S\brac{X,\,Y,\,Z}}}.
\end{align}

It follows from the above that 
\begin{equation}\label{XYZ ineq 1}
    \log \log H\brac{X,\,Y,\,Z} = T <\mathcal{C}_{131} S\brac{X,\,Y,\,Z}\frac{\log \log 3 S\brac{X,\,Y,\,Z}}{\log 3 S\brac{X,\,Y,\,Z}}.
\end{equation}

We recall that by Northcott's Theorem again, we can assume that $S\brac{X,\,Y,\,Z}$ can be larger than any given constant while only dropping finitely many solutions to $X+Y=Z$. Thus, only loosing finitely many solutions, for sufficiently large $S\brac{X,\,Y,\,Z}$ it follows from \eqref{XYZ ineq 1} that
\begin{equation}\label{XYZ ineq 1'}
    T <\mathcal{C}_{132} S\brac{X,\,Y,\,Z}\frac{\log \log  S\brac{X,\,Y,\,Z}}{\log  S\brac{X,\,Y,\,Z}}.
\end{equation}

Taking logarithms on both side of inequality \eqref{XYZ ineq 1'}, we can assume $S\brac{X,\,Y,\,Z}$  is large enough to give that
\begin{align}
    \log T &< \log \brac{\mathcal{C}_{132} S\brac{X,\,Y,\,Z}\frac{\log \log  S\brac{X,\,Y,\,Z}}{\log  S\brac{X,\,Y,\,Z}}} \nonumber \\
    &= \log \mathcal{C}_{132} + \log S\brac{X,\,Y,\,Z}+ \log \log \log S\brac{X,\,Y,\,Z}- \log \log S\brac{X,\,Y,\,Z} \nonumber \\
    &< 2 \log S\brac{X,\,Y,\,Z}.
\end{align}

For ease later, we divide the above by $2$ to give that
\begin{equation}\label{XYZ ineq 2}
    \frac{1}{2}\log T < \log S\brac{X,\,Y,\,Z}.
\end{equation}
For triples $\brac{X,\,Y,\,Z}$ satisfying \eqref{xyz ineq}, taking logarithms in inequality \eqref{xyz ineq} we deduce that 
\begin{align}\label{XYZ ineq 3}
\log S\brac{X,\,Y,\,Z} &< \log \brac{T \frac{\log T}{\log \log T \phi\brac{T}}} \nonumber \\
&= \log T 
+ \log \log T -\log \log \log T -\log \phi\brac{T} \nonumber \\
&< 2 \log T.
\end{align}
We note that this also implies that for sufficiently large $S\brac{X,\,Y,\,Z}$,
\[\log \log S\brac{X,\,Y,\,Z} <2 \log \log T.\]

Substituting this and \eqref{XYZ ineq 2} into \eqref{xyz ineq} we obtain that
\begin{align}\label{XYZ ineq 4}
    S\brac{X,\,Y,\,Z} &< T \frac{\log T}{\log \log T \phi\brac{T}} \nonumber \\
    &< T \frac{2 \log S\brac{X,\,Y,\,Z}}{\frac{1}{2} \log \log \brac{S\brac{X,\,Y,\,Z}} \phi\brac{T}} \nonumber \\
    &=4T \frac{\log S\brac{X,\,Y,\,Z}}{\log \log \brac{S\brac{X,\,Y,\,Z}}\phi\brac{T}}.
\end{align}

Rearranging the above we obtain that
\begin{equation}\label{XYZ ineq 5}
\frac{1}{4}\frac{S\brac{X,\,Y,\,Z} \log \log S\brac{X,\,Y,\,Z}}{\log S\brac{X,\,Y,\,Z}} \phi \brac{T}< T
\end{equation}

We now have two inequalities relating $S$ and $T$, namely \eqref{XYZ ineq 1'} and \eqref{XYZ ineq 5} given above. We compare these directly to find that
\begin{align}
    \frac{1}{4}\frac{S\brac{X,\,Y,\,Z}\log \log S\brac{X,\,Y,\,Z}}{\log S\brac{X,\,Y,\,Z}} \phi \brac{T}< T<\mathcal{C}_{132} \frac{S\brac{X,\,Y,\,Z}\log \log  S\brac{X,\,Y,\,Z}}{\log  S\brac{X,\,Y,\,Z}},
\end{align}
which we rewrite as
\begin{align}
    \frac{S\brac{X,\,Y,\,Z} \log \log S\brac{X,\,Y,\,Z}}{\log S\brac{X,\,Y,\,Z}} \phi \brac{T}< T<\mathcal{C}_{133} \frac{S\brac{X,\,Y,\,Z}\log \log  S\brac{X,\,Y,\,Z}}{\log  S\brac{X,\,Y,\,Z}}.
\end{align}
Cancelling terms on both sides gives us that
\begin{equation}\label{final XYZ ineq}
    \phi\brac{T}< \mathcal{C}_{133}.
\end{equation}
However, $\phi\brac{T}$ tends to $+\infty$ as $T$ tends to $0$, and as $T=\log \log H\brac{X,\,Y,\,Z}$, this happens as $H\brac{X,\,Y,\,Z}$ gets arbitrarily large. Thus, there is a value $B$ such that if $H\brac{X,\,Y,\,Z}>B$, then \eqref{final XYZ ineq} cannot hold. This gives an upper bound for values of $H\brac{X,\,Y,\,Z}$ such that the triple satisfies the assumptions of the theorem. It thus follows by Northcutt's Theorem that there are only finitely many primitive triples $\brac{X,\,Y,\,Z}$ satisfying $\brac{X+Y=Z}$ with
\[S\brac{X,\,Y,\,Z}\leq \log \log H\brac{X,\,Y,\,Z}\frac{\log \log \log H\brac{X,\,Y,\,Z}}{\log \log \log \log H\brac{X,\,Y,\,Z} \phi\brac{\log \log H\brac{X,\,Y,\,Z}}}\]

\end{proof}
We note that we could also directly prove that there are only finitely many primitive integer triples $\brac{X,\,Y,\,Z}$ satisfying $X+Y=Z$ with 
\[S\brac{X,\,Y,\,Z} < c \log \log H\brac{X,\,Y,\,Z}\]
for any constant $c \in \RR,\, c>0$ using the same method of proof as above, though this result follows from Theorem~\ref{XYZ improvement} as stated previously.


\begin{thebibliography}{10}

\bibitem{berstel1976deux}
J.~Berstel and M.~Mignotte.
\newblock Deux propri{\'e}t{\'e}s d{\'e}cidables des suites r{\'e}currentes
  lin{\'e}aires.
\newblock {\em Bulletin de la Soci{\'e}t{\'e} Math{\'e}matique de France},
  104:175--184, 1976.

\bibitem{bombieri2007heights}
E.~Bombieri and W.~Gubler.
\newblock {\em Heights in Diophantine geometry}.
\newblock Number~4. Cambridge university press, 2007.

\bibitem{browkin2000abc}
J.~Browkin.
\newblock The abc-conjecture.
\newblock In {\em Number theory}, pages 75--105. Springer, 2000.

\bibitem{browkin2006abc}
J.~Browkin.
\newblock The abc--conjecture for algebraic numbers.
\newblock {\em Acta Mathematica Sinica}, 22(1):211--222, 2006.

\bibitem{buchmann1987computation}
J.~Buchmann.
\newblock On the computation of units and class numbers by a generalization of
  lagrange's algorithm.
\newblock {\em Journal of Number Theory}, 26(1):8--30, 1987.

\bibitem{bugeaud1996bounds}
Y.~Bugeaud and K.~Gy{\H{o}}ry.
\newblock Bounds for the solutions of unit equations.
\newblock {\em Acta Arithmetica}, 74:67--80, 1996.

\bibitem{childress2008class}
N.~Childress.
\newblock {\em Class field theory}.
\newblock Springer Science \& Business Media, 2008.

\bibitem{everest2003recurrence}
G.~Everest, A.~J. Van Der~Poorten, I.~Shparlinski, T.~Ward, et~al.
\newblock {\em Recurrence sequences}, volume 104.
\newblock American Mathematical Society Providence, RI, 2003.

\bibitem{gyory2010s}
J.E. Evertse and K.~Gy\H{o}ry.
\newblock {\em Unit equations in Diophantine number theory}, volume 146.
\newblock Cambridge University Press, 2015.

\bibitem{gyHory2008abc}
K.~Gy{\H{o}}ry.
\newblock On the $ abc $ conjecture in algebraic number fields.
\newblock {\em Acta Arithmetica}, 133:281--295, 2008.

\bibitem{gyHory2019bounds}
K.~Gy{\H{o}}ry.
\newblock Bounds for the solutions of $ s $-unit equations and decomposable
  form equations {II}.
\newblock {\em Publ.Math.Debrecen}, 94:507--526, 2019.

\bibitem{gyorynewpaper}
K.~Gy{\H{o}}ry.
\newblock {$S$}-unit equations and {M}asser's $abc$-conjecture in algebraic
  number fields.
\newblock {\em Submitted}, 2021.

\bibitem{gyHory2006bounds}
K.~Gy{\H{o}}ry and K.~Yu.
\newblock Bounds for the solutions of $ s $-unit equations and decomposable
  form equations.
\newblock {\em Acta Arithmetica}, 123:9--41, 2006.

\bibitem{halava2005skolem}
V.~Halava, T.~Harju, M.~Hirvensalo and J.~Karhum{\"a}ki.
\newblock Skolem’s problem--on the border between decidability and undecidability.
\newblock {\em Technical Report 683}, Turku Centre for Computer Science, 2005.

\bibitem{harper2016minor}
A.~J. Harper.
\newblock Minor arcs, mean values, and restriction theory for exponential sums
  over smooth numbers.
\newblock {\em Compositio Mathematica}, 152(6):1121--1158, 2016.

\bibitem{lagarias2011smooth}
J.~C. Lagarias and K.~Soundararajan.
\newblock Smooth solutions to the $ abc $ equation: the $ xyz $ conjecture.
\newblock {\em Journal de th{\'e}orie des nombres de Bordeaux}, 23(1):209--234,
  2011.

\bibitem{landau1903neuer}
E.~Landau.
\newblock Neuer beweis des primzahlsatzes und beweis des primidealsatzes.
\newblock {\em Mathematische Annalen}, 56(4):645--670, 1903.

\bibitem{lang2013algebraic}
S.~Lang.
\newblock {\em Algebraic number theory}, volume 110.
\newblock Springer Science \& Business Media, 2013.

\bibitem{le2020tubular}
S.~Le~Fourn.
\newblock Tubular approaches to baker’s method for curves and varieties.
\newblock {\em Algebra \& Number Theory}, 14(3):763--785, 2020.

\bibitem{mason1984diophantine}
R.~C. Mason.
\newblock {\em Diophantine equations over function fields}, volume~96.
\newblock Cambridge University Press, 1984.

\bibitem{masser2002abc}
D.~W. Masser.
\newblock On abc and discriminants.
\newblock {\em Proceedings of the American Mathematical Society},
  130(11):3141--3150, 2002.

\bibitem{masser1985open}
D.~W. Masser.
\newblock Open problems.
\newblock In {\em Proceedings of the symposium on Analytic Number Theory,
  London, 1985}. Imperial College, 1985.

\bibitem{mochizuki2021inter}
S.~Mochizuki.
\newblock Inter-universal teichm{\"u}ller theory I, II, III, IV.
\newblock {\em Publications of the Research Institute for Mathematical
  Sciences}, 57(1):3--723, 2021.

\bibitem{Natarajan2020}
S.~Natarajan and R.~Thangadurai.
\newblock {\em {Pillars of Transcendental Number Theory}}.
\newblock Springer-Verlag, 2020.

\bibitem{neukirch2013algebraic}
J.~Neukirch.
\newblock {\em Algebraic number theory}, volume 322.
\newblock Springer Science \& Business Media, 2013.

\bibitem{oesterle1988nouvelles}
J.~Oesterl{\'e}.
\newblock Nouvelles approches du “th{\'e}oreme” de fermat.
\newblock {\em Ast{\'e}risque}, 161(162):165--186, 1988.

\bibitem{ostafe2020skolem}
A.~Ostafe and I.~E.~Shparlinski.
\newblock On the Skolem problem and some related questions for parametric families of linear recurrence sequences.
\newblock {\em Canadian Journal of Mathematics}, 1--24, 2020.

\bibitem{ouaknine2012decision}
J.~Ouaknine and J.~Worrell.
\newblock Decision problems for linear recurrence sequences.
\newblock In {\em International Workshop on Reachability Problems}, pages
  21--28. Springer, 2012.

\bibitem{rosser1939n}
B.~Rosser.
\newblock The n-th prime is greater than nlogn.
\newblock {\em Proceedings of the London Mathematical Society}, 2(1):21--44,
  1939.

\bibitem{rosser1941explicit}
B.~Rosser.
\newblock Explicit bounds for some functions of prime numbers.
\newblock {\em American Journal of Mathematics}, 63(1):211--232, 1941.

\bibitem{scholze2018abc}
P.~Scholze and J.~Stix.
\newblock Why abc is still a conjecture, 2018.

\bibitem{sha2019effective}
M.~Sha.
\newblock Effective results on the Skolem problem for linear recurrence sequences.
\newblock {\em Journal of Number Theory}, 197:228--249, 2019.

\bibitem{StewartYu}
C.~L. Stewart and K.~Yu.
\newblock On the abc conjecture.
\newblock {\em Mathematische Annalen}, 291:225--230, 1991.

\bibitem{stewart2001abc}
C.~L. Stewart and K.~Yu.
\newblock On the abc conjecture, II.
\newblock {\em Duke Mathematical Journal}, 108(1):169--181, 2001.

\bibitem{surroca2007effectivite}
A.~Surroca.
\newblock Sur l'effectivite du theoreme de Siegel et la conjecture abc.
\newblock {\em Journal of Number Theory}, 124(2):267--290, 2007.

\bibitem{vojta2006diophantine}
P.~A. Vojta.
\newblock {\em Diophantine approximations and value distribution theory},
  volume 1239.
\newblock Springer, 2006.

\bibitem{Waldschmidt}
M.~Waldschmidt.
\newblock {\em {Diophantine Approximation on Linear Algebraic Groups:
  Transcendence Properties of the Exponential Function in Several Variables}}.
\newblock Springer-Verlag, 2000.

\bibitem{yu1989linear}
K.~Yu.
\newblock Linear forms in p-adic logarithms.
\newblock {\em Acta Arithmetica}, 53(2):107--186, 1989.

\bibitem{yu1990linear}
K.~Yu.
\newblock Linear forms in $ p $-adic logarithms. II.
\newblock {\em Compositio Mathematica}, 74(1):15--113, 1990.

\bibitem{yu1994linear}
K.~Yu.
\newblock Linear forms in $ p $-adic logarithms. III.
\newblock {\em Compositio Mathematica}, 91(3):241--276, 1994.

\bibitem{yu1998p}
K.~Yu.
\newblock P-adic logarithmic forms and group varieties I.
\newblock 1998.

\bibitem{yu1999p}
K.~Yu.
\newblock p-adic logarithmic forms and group varieties II.
\newblock {\em Acta Arithmetica}, 89(4):337--378, 1999.

\bibitem{yu2007p}
K.~Yu.
\newblock P-adic logarithmic forms and group varieties III.
\newblock {\em Forum Mathematicum}, 19(2):187--280, 2007.













\end{thebibliography}
\end{document}